\documentclass[11pt]{article}

\usepackage{latexsym,amsfonts}
\usepackage{amssymb,amsthm,upref,amscd}
\usepackage[T1]{fontenc}
\usepackage{times}
\usepackage{amscd}
\usepackage{mathrsfs}
\usepackage[reqno]{amsmath}

\allowdisplaybreaks[4]

\newtheorem{Theorem}{Theorem}[section]
\newtheorem{definition}[Theorem]{Definition}
\newtheorem{lemma}[Theorem]{Lemma}
\newtheorem{proposition}[Theorem]{Proposition}
\newtheorem{corollary}[Theorem]{Corollary}
\newtheorem{remark}[Theorem]{Remark}
\newtheorem*{Thmwn}{Theorem A}

\makeatletter \@addtoreset{equation}{section} \makeatother

\begin{document}

\title{\bf Existence and qualitative properties of solutions for a Choquard-type equation
with Hardy potential
\thanks{Corresponding author. E-mail address: g289623856@126.com (T. Guo),
tangxh@mail.csu.edu.cn (X. H. Tang).}}

\author{Ting Guo$^{\rm a}$, \; Xianhua Tang$^{\rm b,}$$^*$ }

\date{\small\it$^{\rm a}$Faculty of Science, Jiangxi University of Science and Technology, Ganzhou 341000, China\\
$^{\rm b}$School of Mathematics and Statistics, Central South University, Changsha 410083, China}

\maketitle

\begin{minipage}{13cm}
{\small {\bf Abstract:}  In this paper, we study the existence and qualitative properties of positive
solutions to a Choquard-type equation with Hardy potential. We develop a nonlocal version of concentration-compactness principle
involving the Hardy potential to study the existence and
 the asymptotic behavior of positive solutions by transforming
the original problem into a new nonlocal problem in the weighted Sobolev space.
Moreover, we obtain the symmetry of solutions by using the moving
plane method.


}
\end{minipage}

\section{Introduction and main results}

Recently, the nonlinear Choquard equation
\begin{equation}\label{PCA}
 -\Delta u+V(x)u=(I_{\alpha}\ast|u|^{p})|u|^{p-2}u,~~~~~x\in \mathbb{R}^{N},\\
\end{equation}
 where $I_{\alpha}:\mathbb{R}^{N}\setminus\{0\}\rightarrow\mathbb{R}$
is the Riesz potential defined by
\begin{equation*}
 I_{\alpha}(x)=\frac{\varGamma(\frac{N-\alpha}{2})}{\varGamma(\frac{\alpha}{2})\pi^{N/2}2^{\alpha}|x|^{N-\alpha}},
 \end{equation*}
has received quite a few investigation. Physical motivation of (\ref{PCA}) comes from the special case
\begin{equation}\label{TSPCA}
 -\Delta u+u=(I_{2}\ast|u|^{2})u,~~~~~x\in \mathbb{R}^{3}.\\
\end{equation}
 Eq. (\ref{TSPCA}) is called the Choquard-Pekar equation \cite{EHLI,SPEK}, Hartree equation \cite{LNR} or
Schr\"odinger-Newton equation \cite{MPTO} respectively, according to its physical background and derivation. In the pioneering
work \cite{EHLI}, the existence and uniqueness of positive solutions to (\ref{TSPCA})
were first obtained by Lieb. Later, in \cite{PLLI,PLLIONS}, Lions got the existence and multiplicity results of normalized solution of (\ref{PCA})with some restrictions on $V(x)$ such as positivity etc.

To study problem (\ref{PCA})
variationally, the well-known Hardy-Littlewood-Sobolev
inequality is the starting point.
\begin{Thmwn} {\rm(\cite{LIEBLOSS})} {\rm(Hardy-Littlewood-Sobolev inequality)}\  Let $\alpha\in (0,N)$, and $p,r>1$ with
  \begin{equation*}
\frac{1}{p}+\frac{1}{r}=1+\frac{\alpha}{N}.
\end{equation*}
Let $f\in L^{p}(\mathbb{R}^{N})$, $g\in L^{r}(\mathbb{R}^{N})$. Then there exists a sharp constant
$S(N,\alpha,p,r)$, independent of $f$ and $g$, such that
  \begin{equation}\label{hdlsiyjsz}
  \left|\int_{\mathbb{R}^{N}}\int_{\mathbb{R}^{N}}\frac{f(x)g(y)}{|x-y|^{N-\alpha}}\mathrm{d}x\mathrm{d}y\right|
  \leq
  S(N,\alpha,p,r)|f|_p |g|_r,
\end{equation}
where $|f|_p=\left(\int_{\mathbb{R}^{N}}|f|^p\mathrm{d}x\right)^{1/p}$. If $p=r=\frac{2N}{N+\alpha}$, then
$$
S(N,\alpha,p,r)=S(N,\alpha)=\pi^{\frac{N-\alpha}{2}}
  \frac{\varGamma(\frac{\alpha}{2})}{\varGamma(\frac{N+\alpha}{2})}
  \left(\frac{\varGamma(\frac{N}{2})}{\varGamma(N)}\right)^{-{\frac{\alpha}{N}}}.
$$
In this case, the equality in {\rm (\ref{hdlsiyjsz})} is achieved if and only if
$f\equiv (const.)g$ and
$$
g(x)=A(t^2+|x-\xi|^2)^{-\frac{N+\alpha}{2}}
$$
for some $A\in \mathbb{C}$, $\xi \in \mathbb{R}^{N}$ and $0\neq t\in \mathbb{R}$.
\end{Thmwn}

For $N\geq 3$, $\alpha\in (0,N)$. By the Sobolev embedding theorem,
$H^{1}(\mathbb{R}^{N})\subset L^{\frac{2Np}{N+\alpha}}(\mathbb{R}^{N})$ if and only if
$p\in [\frac{N+\alpha}{N}, \frac{N+\alpha}{N-2}]$.
Usually,
$\frac{N+\alpha}{N-2}$ (or $ \frac{N+\alpha}{N}$)
is called the upper (or lower) critical exponent with
respect to the Hardy-Littlewood-Sobolev inequality
(see \cite{LIEBLOSS}). For the subcritical autonomous case
$V(x)\equiv1$ and $p\in(\frac{N+\alpha}{N}, \frac{N+\alpha}{N-2})$,
Moroz and
Van Schaftingen \cite{MOVS} obtained some nonexistence results
under the range $p\geq\frac{N+\alpha}{N-2}$ or
$p\leq\frac{N+\alpha}{N}$, as well as established the
regularity, positivity and the decay estimates of
groundstates.
There are not a few approaches devoted to the existence and multiplicity of solutions of \eqref{PCA}
and their qualitative properties(see for instance \cite{GvanS,MAZHAO,MOVS,MS2}).

In the present paper, we consider the upper critical problem
\begin{equation}\label{eq1}
-\Delta u-\frac{\vartheta}{|x|^2}u=(I_{\alpha}\ast|u|^{\bar{p}})|u|^{\bar{p}-2}u,
~~~~~x\in\mathbb{R}^{N}\setminus\{0\},
\end{equation}
where $N\geq 3$, $0<\vartheta<\frac{(N-2)^2}{4}$, $(N-4)_{+}<\alpha<N$ and $\bar{p}=\frac{N+\alpha}{N-2}$.
In view of \cite{RS}, the Hardy potential $\frac{\vartheta}{|x|^2}$ is critical.  It indeed has the same
homogeneity as the
Laplacian and does not belong to the Kato's class.
Hence, problem (\ref{eq1}) can be considered as doubly critical.

When $\vartheta=0$, (\ref{eq1}) reduces to
\begin{equation}\label{eq4}
-\Delta u=(I_{\alpha}\ast|u|^{\bar{p}})|u|^{\bar{p}-2}u,
~~~~~x\in\mathbb{R}^{N}.
\end{equation}
By using the method of moving plane in integral forms, Guo et al. \cite{GUOLUN} proved that
all positive solutions of (\ref{eq4}) are radially symmetric and radially decreasing
about some point $x_{0}\in\mathbb{R}^{N}$ and are given by
 $$u(x)=
 C_{\alpha}\left(\frac{t}{t^2+|x-x_{0}|^2}\right)^\frac{N-2}{2},$$
where $C_{\alpha}$ and $t$ are positive constants.

For the local case, Terracini \cite{Terracini} considered the doubly critical problem
\begin{equation}\label{eq2}
-\Delta u-\frac{\mu}{|x|^2}u=|u|^{2^*-2}u,
~~~~~x\in\mathbb{R}^{N}\setminus\{0\}.\\
\end{equation}
Denote $\beta:=\sqrt{\frac{(N-2)^2}{4}-\mu}$ for $0\leq \mu <\frac{(N-2)^2}{4}$.
The author proved existence, uniqueness and qualitative properties
of the solution of problem (\ref{eq2}) by using the variation arguments and the
moving plane method. In particular, it was shown that all positive
solutions of (\ref{eq2}) are of the form $U_{\mu}^{\varepsilon}(x):=\varepsilon^
{\frac{2-N}{2}}U_{\mu}(x/\varepsilon)$, where
\begin{equation}\label{SHS}
U_{\mu}(x)
:=\frac{C_{\mu}(N)}{|x|^{\frac{N-2}{2}-\beta}(1+|x|^
{\frac{4\beta}{N-2}})^{\frac{N-2}{2}}}
\end{equation}
for an appropriate constant $C_{\mu}(N)>0$. Moreover, these
solutions are minimizers of $$S_{\mu}:=
\inf_{u\in D^{1,2}(\mathbb{R}^{N})\setminus\{0\}}
\frac{ \int_{\mathbb{R}^{N}}(|\nabla u|^2-\frac
{\mu |u|^2}{|x|^2})\mathrm{d}x}{(\int_{\mathbb{R}^{N}}
|u|^{2^*}\mathrm{d}x)^\frac{2}{2^*}},$$
where $D^{1,2}(\mathbb{R}^{N})=\{u\in L^{2^*}(\mathbb{R}^{N}):\
\nabla u \in L^{2}(\mathbb{R}^{N})\}$.

For fractional equations, Dipierro et al. \cite{SDipierro}
investigated the doubly critical problem
\begin{equation}\label{eq3}
(-\Delta)^{s} u=\vartheta\frac{u}{|x|^{2s}}+u^{2_s^*-1},
~~~~~u\in \dot{H}^s(\mathbb{R}^{N}),
\end{equation}
where $\vartheta>0$ is sufficiently small and $N> 2s$, $0<s<1$, $2_s^*=\frac{2N}{N-2s}$.
They obtained the existence and asymptotic behavior of
positive solutions of problem (\ref{eq3}). For more related topics, we refer to \cite{CaoHan,FT,Guo,Han,Smets}
and references therein.

To the best of our knowledge, there are no results on nontrivial solution for problem (\ref{eq1})
 so far. Motivated by \cite{SDipierro,GUOLUN,Terracini} and the aforementioned works, in the
present paper, we will study the existence, symmetry and the asymptotic behavior of positive solutions for (\ref{eq1}).

Now, we are in the position to state the main results of this paper.

\begin{Theorem}\label{HC}
Let $N\geq 3$, $(N-4)_{+}<\alpha<N$ and $0<\vartheta<\frac{(N-2)^2}{4}$.
Then problem {\rm (\ref{eq1})} has a positive radially symmetric solution.
\end{Theorem}

To prove Theorem \ref{HC}, we consider a constrained minimization problem for
the functional
$$
\Phi (u)=\int_{\mathbb{R}^{N}}\left(|\nabla u|^2
 -\vartheta \frac{|u|^2}{|x|^2}\right)
  \mathrm{d}x
$$
on the space $D^{1,2}(\mathbb{R}^{N})$, restricted to the set $\mathcal{M}= \{u\in
 D^{1,2}(\mathbb{R}^{N}):\ \int_{\mathbb{R}^{N}}(I_{\alpha}\ast|u|^{\bar{p}})|u|^{\bar{p}}\mathrm{d}x=1 \}$.
In order to show that the minimizer can be achieved, we must overcome the main difficulty of the
lack of compactness. The Hardy term and the convolution term bring many obstacles
for recovering the compactness. We shall establish a suitable concentration-compactness lemma
to overcome this difficulty.

\begin{remark}\label{wsbs}
Suppose that $0 \leq u \in D^{1,2}(\mathbb{R}^{N})$ be a weak solution of {\rm (\ref{eq1})}. Since
$\bar{p}>2$, the Hardy-Littlewood-Sobolev inequality leads to
$$
\frac{\vartheta}{|x|^2}+(I_{\alpha}\ast|u|^{\bar{p}})|u|^{\bar{p}-2}\in L^{N/2}_{loc}(\mathbb{R}^{N}\setminus\{0\}).
$$
Then from the classical elliptic regularity theory, $u \in C^{2}(\mathbb{R}^{N}\setminus\{0\})$.
\end{remark}

\par\noindent
\begin{Theorem}\label{AB}
Let $N\geq 3$, $(N-4)_{+}<\alpha<N$ and $0<\vartheta<\frac{(N-2)^2}{4}$. Suppose that
$0 \lvertneqq u \in D^{1,2}(\mathbb{R}^{N})$ be a solution of {\rm (\ref{eq1})}. Then
there exist two positive constants $c$ and $C$ such that
\begin{equation}\label{APF}
\frac{c}{|x|^{\beta}\left(1+|x|^{2-\frac{4\beta}{N-2}}\right)^{\frac{N-2}{2}}}
\leq u(x) \leq \frac{C}{|x|^{\beta}\left(1+|x|^{2-\frac{4\beta}{N-2}}\right)^{\frac{N-2}{2}}},
~~~~~x\in\mathbb{R}^{N}\setminus\{0\},
\end{equation}
where
$$
\beta=\frac{N-2-\sqrt{(N-2)^2-4\vartheta}}{2}.
$$
\end{Theorem}

In order to prove Theorem \ref{AB}, we first use the divergence theorem to convert equation (\ref{eq1})
into an equivalent equation in a weighted space; Then we give a regularity result for the new nonlocal
problem. Moreover, a weighted Harnack inequality is established following Di Benedetto and Trudinger's idea \cite{DTR}
by using a density lemma and the Krylov-Safonov covering theorem. Finally, we obtain the asymptotic behavior of the
positive solution of (\ref{eq1}) near the origin and infinity.

\begin{remark}\label{deve}
Comparing formulas {\rm (\ref{SHS})} and {\rm (\ref{APF})}, we see that Theorem \ref{AB} provides an
asymptotic property similar to that in {\rm \cite{Terracini}}.
\end{remark}

\begin{Theorem}\label{SYM}
Let $N\geq 3$, $(N-4)_{+}<\alpha<N$ and $0<\vartheta<\frac{(N-2)^2}{4}$. Suppose that
$0 \lvertneqq u \in D^{1,2}(\mathbb{R}^{N})$ be a solution of {\rm (\ref{eq1})}. Then
$u(x)$ is radially symmetric and monotone decreasing about the origin.
\end{Theorem}

We prove Theorem \ref{SYM} by the moving plane method, which was invented by Alexandrov in the
1950s and has been developed by many researcheres (see \cite{CaGiSp,CLO,Serrin}). For example,
in \cite{CLO}, using this method in integral form to integral equations,
 Chen et al. obtained the symmetry and monotonicity of the solution. The homogeneity of Hardy potential
makes it difficult to directly use the moving plane method in integral form. Thanks to the singularity
of the solution at the origin, we can use the strong maximum principle to overcome this difficulty.

\vskip 1mm
The rest of the paper is organized as follows. In section 2, by establishing
a nonlocal version of concentration-compactness principle involving Hardy term, we are
able to obtain the existence result.
In section 3, the regularity of the solution to the equivalent equation is proved for the first time.
In addition, we prove the weak Harnack inequality by establishing a weighted Poincar\'{e}
inequality. And the asymptotic behavior is obtained. In section 4, via the moving plane method,
the radial symmetry and monotonicity of the solution are proved.

Throughout this paper, we use $c_1, c_2,\cdots$ or $C_1, C_2,\cdots$ to denote
different positive constants.

\section{Existence of extremal functions and proof of Theorem \ref{HC}}
In this section, we give the proof of Theorem \ref{HC}. To prove the existence of a solution
to (\ref{eq1}) we consider the following minimization problem
$$
S_{\vartheta}:=\inf\left\{\int_{\mathbb{R}^{N}}\left(|\nabla u|^2
 -\vartheta \frac{|u|^2}{|x|^2}\right)
  \mathrm{d}x:u\in D^{1,2}(\mathbb{R}^{N}),
~\int_{\mathbb{R}^{N}}(I_{\alpha}\ast|u|^{\bar{p}})|u|^{\bar{p}}\mathrm{d}x=1\right\}.
$$
Using symmetrization techniques, we have
\begin{lemma}\label{SYJXHXL} \  There exists a sequence $\{v_{n}\}\subset D^{1,2}(\mathbb{R}^{N})$
that consist of radial and radially decreasing functions such that
 \begin{equation}\label{minipro}
\int_{\mathbb{R}^{N}}(I_{\alpha}\ast|v_{n}|^{\bar{p}})|v_{n}|^{\bar{p}}\mathrm{d}x=1,
~~~~~~\int_{\mathbb{R}^{N}}\left(|\nabla v_{n}|^2
 -\vartheta \frac{|v_{n}|^2}{|x|^2}\right)
  \mathrm{d}x\rightarrow S_{\vartheta}, ~~~n\rightarrow\infty.
 \end{equation}
\end{lemma}

\begin{proof}[\bf Proof.]
Let $\{u_{n}\}\subset D^{1,2}(\mathbb{R}^{N})$ be a minimizing sequence satisfying
$$
\int_{\mathbb{R}^{N}}(I_{\alpha}\ast|u_{n}|^{\bar{p}})|u_{n}|^{\bar{p}}\mathrm{d}x=1,
~~~~~~\int_{\mathbb{R}^{N}}\left(|\nabla u_{n}|^2
 -\vartheta \frac{|u_{n}|^2}{|x|^2}\right)
  \mathrm{d}x\rightarrow S_{\vartheta}, ~~~n\rightarrow\infty.
$$
By Riesz's rearrangement inequality \cite{EHLI}, one has
$$
\lambda_n:=\left(\int_{\mathbb{R}^{N}}(I_{\alpha}\ast|u_{n}^*|^{\bar{p}})|u_{n}^*|^{\bar{p}}\mathrm{d}x\right)
^{1/2\bar{p}}\geq \left(\int_{\mathbb{R}^{N}}(I_{\alpha}\ast|u_{n}|^{\bar{p}})|u_{n}|^{\bar{p}}\mathrm{d}x\right)
^{1/2\bar{p}}=1,
$$
where $u_{n}^*$ be the symmetric decreasing rearrangement of $u_{n}$. Define $v_n:=\lambda_n^{-1}u_n^*$, we have
$$\int_{\mathbb{R}^{N}}(I_{\alpha}\ast|v_{n}|^{\bar{p}})|v_{n}|^{\bar{p}}\mathrm{d}x=1.$$
By using the symmetric rearrangement inequality \cite{LIEBLOSS}, we get
$$\int_{\mathbb{R}^{N}}\frac{|u_{n}|^2}{|x|^2}
  \mathrm{d}x \leq \int_{\mathbb{R}^{N}}\frac{|u_{n}^*|^2}{|x|^2}
  \mathrm{d}x.
$$
Thanks to the P\'{o}lya-Szeg\"{o} inequality, we obtain
 \begin{equation*}
 \aligned
 S_{\vartheta}&= S_{\vartheta}\left(\int_{\mathbb{R}^{N}}(I_{\alpha}\ast|v_{n}|^{\bar{p}})|v_{n}|^{\bar{p}}\mathrm{d}x
 \right)^{1/\bar{p}}\leq \lambda_n^{-2}\int_{\mathbb{R}^{N}}\left(|\nabla u_{n}^*|^2
 -\vartheta \frac{|u_{n}^*|^2}{|x|^2}\right)\mathrm{d}x
  \\&\leq \lambda_n^{-2}\int_{\mathbb{R}^{N}}\left(|\nabla u_{n}|^2
 -\vartheta \frac{|u_{n}|^2}{|x|^2}\right) \mathrm{d}x
 \leq \int_{\mathbb{R}^{N}}\left(|\nabla u_{n}|^2
 -\vartheta \frac{|u_{n}|^2}{|x|^2}\right) \mathrm{d}x.
 \endaligned
\end{equation*}
Thus,
$$
\lambda_n \rightarrow 1,~~~~~\int_{\mathbb{R}^{N}}\left(|\nabla v_{n}|^2
 -\vartheta \frac{|v_{n}|^2}{|x|^2}\right)
  \mathrm{d}x\rightarrow S_{\vartheta}.
$$
\end{proof}

\begin{lemma}\label{SLDL} \  Let $0\leq\xi \in C_{c}^{1}(\mathbb{R}^{N})$. If
$\omega_n\rightharpoonup0~\mbox{in}~D^{1,2}(\mathbb{R}^{N})$, then there holds
\begin{equation*}
  |I_{\alpha}\ast (\xi|\omega_n|^{\bar{p}})-(I_{\alpha}\ast |\omega_n|^{\bar{p}})\xi|_\frac{2N}{N-\alpha}=o(1).
\end{equation*}
\end{lemma}

\begin{proof}[\bf Proof.]
For $x\in \mathbb{R}^{N}$, we write
\begin{equation*}
 \aligned
 \varphi_{n}(x)&=I_{\alpha}\ast (\xi|\omega_n|^{\bar{p}})-(I_{\alpha}\ast |\omega_n|^{\bar{p}})\xi
 =\int_{\mathbb{R}^{N}}\frac{|\omega_n(y)|^{\bar{p}}(\xi(y)-\xi(x))}{|x-y|^{N-\alpha}}
  \mathrm{d}y
 \\&=:\int_{\mathbb{R}^{N}}|\omega_n(y)|^{\bar{p}}T(x,y)
  \mathrm{d}y.
 \endaligned
\end{equation*}
Let
$$
Q(x)=\left\{
 \begin{aligned}
 & \frac{1}{|x|^{N-\alpha-1}},& |x|\leq 1; \\
 & \frac{1}{|x|^{N-\alpha}},& |x|>1.\\
 \end{aligned}
 \right.
$$
Since $\xi \in C_{c}^{1}(\mathbb{R}^{N})$, we see that $|T(x,y)|\leq C Q(x-y)$, and so
$|\varphi_n|\leq C(Q\ast |\omega_n|^{\bar{p}})$. Note that $Q\in L^{\frac{N}{N-\alpha-\varepsilon}}(\mathbb{R}^{N})$
for $\varepsilon>0$ sufficiently small. We deduce from Young's inequality that
 \begin{equation}\label{esin}
 |\varphi_n|_\frac{2N}{N-\alpha-2\varepsilon}\leq
 C |Q\ast |\omega_n|^{\bar{p}}|_\frac{2N}{N-\alpha-2\varepsilon}\leq
 C |Q|_\frac{N}{N-\alpha-\varepsilon}|\omega_n|_{2^*}^{\bar{p}}\leq
 C.
 \end{equation}
From $\omega_n\rightharpoonup0~\mbox{in}~D^{1,2}(\mathbb{R}^{N})$, we see up to a subsequence that
$\omega_n \rightarrow 0~\mbox{a.e.~on}~\mathbb{R}^N$ and the sequence $\{|\omega_n|^{\bar{p}}\}$
is bounded in $L^{\frac{2N}{N+\alpha}}(\mathbb{R}^{N})$.
Hence, $|\omega_n|^{\bar{p}}\rightharpoonup0\
\mbox{in}~ L^{\frac{2N}{N+\alpha}}(\mathbb{R}^{N})$.
Then by the Hardy-Littlewood-Sobolev inequality,
we have that
$
I_{\alpha}\ast |\omega_n|^{\bar{p}}\rightharpoonup0\
\mbox{in}~ L^{\frac{2N}{N-\alpha}}(\mathbb{R}^{N}).
$
Note that $I_{\alpha}\ast |\omega_n|^{\bar{p}}\geq 0$.
Therefore, for every $R>0$,
$$
\int_{B_{R}(0)} I_{\alpha}\ast |\omega_n|^{\bar{p}}\mathrm{d}x
=\int_{\mathbb{R}^{N}}( I_{\alpha}\ast |\omega_n|^{\bar{p}}) \chi_{B_{R}(0)}\mathrm{d}x
\rightarrow 0~~~\mbox{as}~n\rightarrow \infty.$$
Namely, $I_{\alpha}\ast |\omega_n|^{\bar{p}}\rightarrow0\ \mbox{in}\ L^{1}_{loc}(\mathbb{R}^{N})$.
Up to a subsequence, one gets
$I_{\alpha}\ast |\omega_n|^{\bar{p}}\rightarrow0$
a.e. on $\mathbb{R}^{N}$.
Similarly, $I_{\alpha}\ast (\xi|\omega_n|^{\bar{p}})\rightarrow0$
a.e. on $\mathbb{R}^{N}$.
Hence, from (\ref{esin}), we have that
\begin{equation}\label{LOCC}
 \varphi_{n}\rightarrow0~~~
\mbox{in}~ L^{\frac{2N}{N-\alpha}}_{loc}(\mathbb{R}^{N}).
\end{equation}
Suppose that $supp \xi \subset B_{r}(0)$. For $R>r$, let $H_{R}(x)=I_{\alpha} \chi_{B_{R-r}^{c}(0)}$.
It follows from Young's inequality and the H\"{o}lder inequality that
 \begin{equation*}
 \aligned
 |I_{\alpha}\ast (\xi|\omega_n|^{\bar{p}})|_{\frac{2N}{N-\alpha},B_R^{c}(0)}
 &\leq |H_R\ast (\xi|\omega_n|^{\bar{p}})|_{\frac{2N}{N-\alpha}}
  \\& \leq |H_R|_{\frac{2N}{N-\alpha}} |\xi|\omega_n|^{\bar{p}}|_{1}
  \\& \leq C |H_R|_{\frac{2N}{N-\alpha}}|\omega_n|_{2^*}^{\bar{p}}r^{\frac{N-\alpha}{2}}.
 \endaligned
\end{equation*}
Since $\lim\limits_{R\rightarrow\infty}|H_R|_{\frac{2N}{N-\alpha}}=0$, we see that
$$
\lim\limits_{R\rightarrow\infty} |I_{\alpha}\ast
(\xi|\omega_n|^{\bar{p}})|_{\frac{2N}{N-\alpha},B_R^{c}(0)}=0~~~\mbox{uniformly~ for}~n \in \mathbb{N}.
$$
Observe that $|(I_{\alpha}\ast
|\omega_n|^{\bar{p}})\xi|_{\frac{2N}{N-\alpha},B_R^{c}(0)}=0$. Thus,
$$
\lim\limits_{R\rightarrow\infty} |\varphi_n|_{\frac{2N}{N-\alpha},B_R^{c}(0)}=0~~~\mbox{uniformly~ for}~n \in \mathbb{N},
$$
which, together with (\ref{LOCC}), yields the conclusion.
\end{proof}

We recall the second concentration-compactness principle of Lions.
\begin{lemma}\label{MESLM} {\rm(\cite{PLLIONSLIM})} \ Let $\mu$, $\nu$ be two bounded nonnegative measures on
$\mathbb{R}^{N}$ satisfying for some constant $C\geq 0$,
$$
\left(\int_{\mathbb{R}^{N}}|\xi|^{q}
  \mathrm{d}\nu\right)^{1/q}
  \leq C\left(\int_{\mathbb{R}^{N}}|\xi|^{p}
  \mathrm{d}\mu\right)^{1/p},~~~\forall~\xi\in  C_{c}^{\infty}(\mathbb{R}^{N}),
$$
where $1\leq p <q\leq \infty$. Then there exist an at most countable index set $\Lambda$,
sequences $\{x_i\}\subset \mathbb{R}^{N}$, $\{\nu_i\}\subset (0,\infty)$,
such that
$$
\nu=\sum_{i\in \Lambda}\nu_i \delta_{x_i},~~~\mu\geq C^{-p}\sum_{i\in \Lambda}\nu_i^{p/q} \delta_{x_i},
$$
where $\delta_{x_i}$ is the Dirac mass at $x_i \in \mathbb{R}^{N}$.
\end{lemma}

Inspired by the ideas in \cite{Chabrowski,Willem}, we can establish the following
concentration-compactness lemma, which will play a crucial role
in proving the existence of positive solutions for (\ref{eq1}).

\begin{lemma}\label{CCLMCH} \  Let $\{u_{n}\}\subset D^{1,2}(\mathbb{R}^{N})$ be a
sequence such that
  $$u_n\rightharpoonup u~~~\mbox{in}~D^{1,2}(\mathbb{R}^{N}),~~~~~~
  u_n \rightarrow u~~~\mbox{a.e.~on}~\mathbb{R}^N.$$
Assume that
$$|\nabla (u_{n}-u)|^2
 -\vartheta \frac{|u_{n}-u|^2}{|x|^2}\rightharpoonup \mu~~~\mbox{and}~~~(I_{\alpha}\ast
|u_{n}-u|^{\bar{p}})|u_{n}-u|^{\bar{p}} \rightharpoonup \nu$$
weakly in the sense of measure and define
 $$\mu_\infty=\lim\limits_{R\rightarrow\infty}
 \varlimsup\limits_{n\rightarrow\infty}
  \int_{|x|\geq R}\left(|\nabla u_{n}|^2
 -\vartheta \frac{|u_{n}|^2}{|x|^2}\right) \mathrm{d}x$$
and
 $$\nu_\infty=\lim\limits_{R\rightarrow\infty}
 \varlimsup\limits_{n\rightarrow\infty}\int_{|x|\geq R}(I_{\alpha}\ast
|u_{n}|^{\bar{p}})|u_{n}|^{\bar{p}}\mathrm{d}x.$$
Then there exist an at most countable index set $\Lambda$,
sequences $\{x_i\}\subset \mathbb{R}^{N}$, $\{\nu_i\}\subset (0,\infty)$,
such that
\begin{equation} \label{HMINE1}
\nu=\sum_{i\in \Lambda}\nu_i \delta_{x_i},~~~\mu\geq S_{\vartheta}\sum_{i\in \Lambda}\nu_i^{1/\bar{p}} \delta_{x_i},
  \end{equation}
\begin{equation} \label{HMINE2}
\mu_\infty\geq S_{\vartheta}\nu_\infty^{1/\bar{p}},
  \end{equation}
\begin{equation} \label{HMINE3}
\mu_\infty+\int_{\mathbb{R}^{N}}\mathrm{d}\mu
+\int_{\mathbb{R}^{N}}\left(|\nabla u|^2
 -\vartheta \frac{|u|^2}{|x|^2}\right) \mathrm{d}x=
 \varlimsup\limits_{n\rightarrow\infty}
  \int_{\mathbb{R}^{N}}\left(|\nabla u_{n}|^2
 -\vartheta \frac{|u_{n}|^2}{|x|^2}\right) \mathrm{d}x
  \end{equation}
and
\begin{equation} \label{HMINE4}
\nu_\infty+\int_{\mathbb{R}^{N}}\mathrm{d}\nu
+\int_{\mathbb{R}^{N}}(I_{\alpha}\ast
|u|^{\bar{p}})|u|^{\bar{p}} \mathrm{d}x=
 \varlimsup\limits_{n\rightarrow\infty}
  \int_{\mathbb{R}^{N}}(I_{\alpha}\ast
|u_{n}|^{\bar{p}})|u_{n}|^{\bar{p}} \mathrm{d}x,
  \end{equation}
where $\delta_{x_i}$ is the Dirac mass at $x_i \in \mathbb{R}^{N}$.
\end{lemma}
\begin{proof}[\bf Proof.]
For $R>0$, let  $\psi_R\in C^{1}\left(\mathbb{R}^{N},[0,1]\right)$ be such that
$\psi_R=1$ for $x\in\mathbb{R}^{N}\setminus  B_{R+1}(0)$ and $\psi_R=0$ for
$x\in B_{R}(0)$. Set $\omega_n=u_n-u$. Then $\omega_n\rightharpoonup 0$ in
$D^{1,2}(\mathbb{R}^{N})$,
$|\nabla \omega_{n}|^2
 -\vartheta \frac{|\omega_n|^2}{|x|^2}\rightharpoonup \mu$
weakly in the sense of measure, hence
\begin{equation*}
\aligned
\int_{\mathbb{R}^{N}}\left(|\nabla u_{n}|^2
 -\vartheta \frac{|u_{n}|^2}{|x|^2}\right) \mathrm{d}x
 &=
 \int_{\mathbb{R}^{N}}\psi_R\left(|\nabla u_{n}|^2
 -\vartheta \frac{|u_{n}|^2}{|x|^2}\right) \mathrm{d}x
  +\int_{\mathbb{R}^{N}}(1-\psi_R)\mathrm{d}\mu
\\&~~~~~+\int_{\mathbb{R}^{N}}(1-\psi_R)\left(|\nabla u|^2
 -\vartheta \frac{|u|^2}{|x|^2}\right) \mathrm{d}x +o(1).
 \endaligned
  \end{equation*}
Therefore,
\begin{equation} \label{ESJXIDT}
\aligned
\varlimsup\limits_{n\rightarrow\infty}
\int_{\mathbb{R}^{N}}\left(|\nabla u_{n}|^2
 -\vartheta \frac{|u_{n}|^2}{|x|^2}\right) \mathrm{d}x
 &=
 \lim\limits_{R\rightarrow\infty}
 \varlimsup\limits_{n\rightarrow\infty}\int_{\mathbb{R}^{N}}\psi_R\left(|\nabla u_{n}|^2
 -\vartheta \frac{|u_{n}|^2}{|x|^2}\right) \mathrm{d}x
 \\&~~~~~ +\int_{\mathbb{R}^{N}}\mathrm{d}\mu
+\int_{\mathbb{R}^{N}}\left(|\nabla u|^2
 -\vartheta \frac{|u|^2}{|x|^2}\right) \mathrm{d}x .
 \endaligned
  \end{equation}
Let $\varphi_R\in C\left(\mathbb{R}^{N},[0,1]\right)$ be such that
$\varphi_R=1$ for $x\in B_{R+1}(0)\setminus  B_{R}(0)$ and $\psi_R=0$ for
$x\in B_{R-1}(0)\cup B_{R+2}^{c}(0)$.
Since $\int_{\mathbb{R}^{N}}\mathrm{d}\mu<\infty$, for $\varepsilon >0$,
there exists $R(\varepsilon) >0$ such that
$$\lim\limits_{n\rightarrow\infty}\int_{\mathbb{R}^{N}}\varphi_R\left(|\nabla \omega_{n}|^2
 -\vartheta \frac{|\omega_{n}|^2}{|x|^2}\right) \mathrm{d}x
 =\int_{\mathbb{R}^{N}}\varphi_R\mathrm{d}\mu<\varepsilon/2,~~~\forall R>R(\varepsilon).
 $$
 Thus, for any $R>R(\varepsilon)$, there exists $N(R,\varepsilon)>0$ such that
 $$\left|\int_{\mathbb{R}^{N}}\varphi_R\left(|\nabla \omega_{n}|^2
 -\vartheta \frac{|\omega_{n}|^2}{|x|^2}\right) \mathrm{d}x\right|
 <\varepsilon,~~~\forall R>R(\varepsilon)~\mbox{and}~n>N(R,\varepsilon).
 $$
 Since $\{\omega_n\}$ is bounded in $D^{1,2}(\mathbb{R}^{N})$. By taking
 $R\rightarrow \infty$, we have that
 $$
\int_{R-1<|x|<R+2} \frac{|\omega_{n}|^2}{|x|^2} \mathrm{d}x
 \leq C \left(\int_{R-1<|x|<R+2} |\omega_{n}|^{2^*} \mathrm{d}x\right)
  ^{\frac{N-2}{N}}\left(\ln \frac{R+2}{R-1}\right)^{\frac{2}{N}}
\rightarrow 0$$
 uniformly for $n \in \mathbb{N}$.
 Then we have for $R>R(\varepsilon)$ and $n>N(R,\varepsilon)$,
 $$
 \int_{{R<|x|<R+1}}|\nabla u_{n}|^2 \mathrm{d}x
 \leq 2\left(\int_{{R<|x|<R+1}}|\nabla \omega_{n}|^2 \mathrm{d}x+
 \int_{{R<|x|<R+1}}|\nabla u|^2 \mathrm{d}x\right)<3\varepsilon.
 $$
 Note that
$$\int_{R<|x|<R+1} \frac{|\omega_{n}|^2}{|x|^2} \mathrm{d}x
 \leq C \left(\ln \frac{R+1}{R}\right)^{\frac{2}{N}}
\rightarrow 0~~~\mbox{as}~R\rightarrow \infty.$$
Hence for $R>R(\varepsilon)$, we get
$$
\left|\varlimsup\limits_{n\rightarrow\infty}
  \int_{|x|\geq R}\left(|\nabla u_{n}|^2
 -\vartheta \frac{|u_{n}|^2}{|x|^2}\right) \mathrm{d}x
 -\varlimsup\limits_{n\rightarrow\infty}\int_{\mathbb{R}^{N}}\psi_R\left(|\nabla u_{n}|^2
 -\vartheta \frac{|u_{n}|^2}{|x|^2}\right) \mathrm{d}x\right|<4\varepsilon,
 $$
 which implies
 $$
 \mu_\infty= \lim\limits_{R\rightarrow\infty}
 \varlimsup\limits_{n\rightarrow\infty}\int_{\mathbb{R}^{N}}\psi_R\left(|\nabla u_{n}|^2
 -\vartheta \frac{|u_{n}|^2}{|x|^2}\right) \mathrm{d}x.
 $$
 Identity (\ref{HMINE3}) follows then from (\ref{ESJXIDT}).
 By the Br\'{e}zis-Lieb lemma, we get
\begin{equation*}
\aligned
\varlimsup\limits_{n\rightarrow\infty}\int_{\mathbb{R}^{N}}
 (I_{\alpha}\ast |u_n|^{\bar{p}})|u_n|^{\bar{p}}\mathrm{d}x
&=\varlimsup\limits_{n\rightarrow\infty}\int_{\mathbb{R}^{N}}
 \psi_R(I_{\alpha}\ast |u_n|^{\bar{p}})|u_n|^{\bar{p}}\mathrm{d}x
+\int_{\mathbb{R}^{N}}(1-\psi_R)\mathrm{d}\nu
\\&~~~~~+\int_{\mathbb{R}^{N}}(1-\psi_R)(I_{\alpha}\ast |u|^{\bar{p}})|u|^{\bar{p}}\mathrm{d}x.
\endaligned
\end{equation*}
Let $R\rightarrow\infty$, we obtain by Lebesgue's theorem,
\begin{equation*}
 \varlimsup\limits_{n\rightarrow\infty}
  \int_{\mathbb{R}^{N}}(I_{\alpha}\ast
|u_{n}|^{\bar{p}})|u_{n}|^{\bar{p}} \mathrm{d}x=\nu_\infty+\int_{\mathbb{R}^{N}}\mathrm{d}\nu
+\int_{\mathbb{R}^{N}}(I_{\alpha}\ast
|u|^{\bar{p}})|u|^{\bar{p}} \mathrm{d}x.
\end{equation*}
According to Lemma \ref{SLDL}, we have for every $0\leq\xi \in C_{c}^{\infty}(\mathbb{R}^{N})$,
$$
\int_{\mathbb{R}^{N}}(I_{\alpha}\ast |\xi\omega_n|^{\bar{p}})|\xi\omega_n|^{\bar{p}}\mathrm{d}x
\rightarrow \int_{\mathbb{R}^{N}}|\xi|^{2\bar{p}}\mathrm{d}\nu
~~~\mbox{as}~n\rightarrow \infty.
$$
Since $\omega_n\rightarrow 0$ in
$L^{2}_{loc}(\mathbb{R}^{N})$, there holds
$$\int_{\mathbb{R}^{N}}\left(|\nabla (\xi\omega_{n})|^2
 -\vartheta \frac{|\xi\omega_{n}|^2}{|x|^2}\right) \mathrm{d}x
\rightarrow \int_{\mathbb{R}^{N}}\xi^{2}\mathrm{d}\mu
~~~\mbox{as}~n\rightarrow \infty.
 $$
By the definition of $S_{\vartheta}$, one has
$$S_{\vartheta}\left(\int_{\mathbb{R}^{N}}(I_{\alpha}\ast |\xi\omega_n|
^{\bar{p}})|\xi\omega_n|^{\bar{p}}\mathrm{d}x\right)^{1/\bar{p}}
\leq \int_{\mathbb{R}^{N}}\left(|\nabla (\xi\omega_{n})|^2
 -\vartheta \frac{|\xi\omega_{n}|^2}{|x|^2}\right) \mathrm{d}x.
 $$
Therefore,
$$
S_{\vartheta}^{1/2}\left(\int_{\mathbb{R}^{N}}|\xi|^{2\bar{p}}\mathrm{d}\nu\right)^{1/2\bar{p}}
\leq \left(\int_{\mathbb{R}^{N}}\xi^{2}\mathrm{d}\mu\right)^{1/2}.
$$
Hence, Lemma \ref{MESLM} shows that (\ref{HMINE1}) holds.
From Lemma \ref{SLDL} and the Br\'{e}zis-Lieb lemma, we have
\begin{equation} \label{BLC}
\aligned
&\varlimsup\limits_{n\rightarrow\infty}\int_{\mathbb{R}^{N}}
 (I_{\alpha}\ast |\psi_R\omega_n|^{\bar{p}})|\psi_R\omega_n|^{\bar{p}}\mathrm{d}x
=\varlimsup\limits_{n\rightarrow\infty}\int_{\mathbb{R}^{N}}
 (I_{\alpha}\ast |\omega_n|^{\bar{p}})|\omega_n|^{\bar{p}}|\psi_R|^{2\bar{p}}\mathrm{d}x
\\&~~~~~=\varlimsup\limits_{n\rightarrow\infty}\int_{\mathbb{R}^{N}}
 (I_{\alpha}\ast |u_n|^{\bar{p}})|u_n|^{\bar{p}}|\psi_R|^{2\bar{p}}\mathrm{d}x
 -\int_{\mathbb{R}^{N}}
 (I_{\alpha}\ast |u|^{\bar{p}})|u|^{\bar{p}}|\psi_R|^{2\bar{p}}\mathrm{d}x.
\endaligned
\end{equation}
On the other hand, by $\omega_n\rightarrow 0$ in
$L^{2}_{loc}(\mathbb{R}^{N})$, we obtain
\begin{equation} \label{BLCRSL}
\aligned
&\varlimsup\limits_{n\rightarrow\infty}\int_{\mathbb{R}^{N}}\left(|\nabla (\psi_R\omega_{n})|^2
 -\vartheta \frac{|\psi_R\omega_{n}|^2}{|x|^2}\right) \mathrm{d}x
\\&~~~=\varlimsup\limits_{n\rightarrow\infty}\int_{\mathbb{R}^{N}}\left(|\nabla u_{n}|^2
 -\vartheta \frac{|u_{n}|^2}{|x|^2}\right)|\psi_R|^2 \mathrm{d}x
 +\vartheta\int_{\mathbb{R}^{N}}\frac{|u|^2}{|x|^2}|\psi_R|^2 \mathrm{d}x
 -\int_{\mathbb{R}^{N}}|\nabla u|^2|\psi_R|^2 \mathrm{d}x.
\endaligned
\end{equation}
By (\ref{BLC}), (\ref{BLCRSL}) and the definition of $S_{\vartheta}$, we deduce
$$
\mu_\infty\geq S_{\vartheta}\nu_\infty^{1/\bar{p}}.
$$
\end{proof}

Now, we are ready to prove the existence of extremal function of the minimizing problem.
\begin{proof}[\bf Proof of Theorem \ref{HC}.] In view of Lemma \ref{SYJXHXL},
we find a minimizing sequence
$\{v_n\}$
that consist of  radially decreasing functions, i.e., $v_n(x)=v_n(|x|)$
for any $x\in \mathbb{R}^{N}$. From (\ref{minipro}), there exists $t_n>0$ such that
$$
\int_{B_{t_n}(0)}(I_{\alpha}\ast|v_{n}|^{\bar{p}})|v_{n}|^{\bar{p}}\mathrm{d}x=\frac{1}{2}.
$$
Set $u_n(x)=t_n^{\frac{N-2}{2}}v_n(t_{n}x)$. Then
\begin{equation} \label{BLCQQ}
\int_{B_{1}(0)}(I_{\alpha}\ast|u_{n}|^{\bar{p}})|u_{n}|^{\bar{p}}\mathrm{d}x=\frac{1}{2},
\end{equation}
and
$$
\int_{\mathbb{R}^{N}}(I_{\alpha}\ast|u_{n}|^{\bar{p}})|u_{n}|^{\bar{p}}\mathrm{d}x=1,
~~~~~~\int_{\mathbb{R}^{N}}\left(|\nabla u_{n}|^2
 -\vartheta \frac{|u_{n}|^2}{|x|^2}\right)
  \mathrm{d}x\rightarrow S_{\vartheta}.
$$
Since $\{u_n\}$ is bounded in $D^{1,2}(\mathbb{R}^{N})$, we may assume, up to a subsequence,
$$u_n\rightharpoonup u~~~\mbox{in}~D^{1,2}(\mathbb{R}^{N}),~~~
u_n \rightarrow u~~~\mbox{a.e.~on}~\mathbb{R}^N,$$
$$|\nabla (u_{n}-u)|^2
 -\vartheta \frac{|u_{n}-u|^2}{|x|^2}\rightharpoonup \mu~~~\mbox{and}~~~(I_{\alpha}\ast
|u_{n}-u|^{\bar{p}})|u_{n}-u|^{\bar{p}} \rightharpoonup \nu$$
weakly in the sense of measure. We deduce from Lemma \ref{CCLMCH} and the definition of $S_\vartheta$ that
\begin{equation*}
\aligned
S_\vartheta&=S_\vartheta\left(\nu_\infty+\int_{\mathbb{R}^{N}}\mathrm{d}\nu
+\int_{\mathbb{R}^{N}}(I_{\alpha}\ast
|u|^{\bar{p}})|u|^{\bar{p}} \mathrm{d}x\right)
\\&\leq
S_\vartheta\nu_\infty^{1/\bar{p}}+S_\vartheta\left(\int_{\mathbb{R}^{N}}\mathrm{d}\nu\right)^{1/\bar{p}}
+S_\vartheta\left(\int_{\mathbb{R}^{N}}(I_{\alpha}\ast
|u|^{\bar{p}})|u|^{\bar{p}} \mathrm{d}x\right)^{1/\bar{p}}
\\&\leq
\mu_\infty+\int_{\mathbb{R}^{N}}\mathrm{d}\mu
+\int_{\mathbb{R}^{N}}\left(|\nabla u|^2
 -\vartheta \frac{|u|^2}{|x|^2}\right) \mathrm{d}x
=S_\vartheta,
\endaligned
\end{equation*}
which implies that $\nu_\infty$, $\int_{\mathbb{R}^{N}}\mathrm{d}\nu$ and
$\int_{\mathbb{R}^{N}}(I_{\alpha}\ast |u|^{\bar{p}})|u|^{\bar{p}} \mathrm{d}x$ are equal
either to $0$ or to $1$.
By (\ref{BLCQQ}), $\nu_\infty\leq 1/2$, so $\nu_\infty=0$. If $\int_{\mathbb{R}^{N}}\mathrm{d}\nu=1$, then
$u=0$. By Lemma \ref{CCLMCH}, we have that $\nu=\sum_{i\in \Lambda}\nu_i \delta_{x_i}$. Note that
$(I_{\alpha}\ast |u_n|^{\bar{p}})|u_n|^{\bar{p}}$ is radial, we know that $\nu$ is concentrated at the origin.
Thus, (\ref{BLC}) leads to
$$
\frac{1}{2}=\lim\limits_{n\rightarrow\infty}\int_{B_{1}(0)}(I_{\alpha}\ast
|u_{n}-u|^{\bar{p}})|u_{n}-u|^{\bar{p}}\mathrm{d}x=\int_{\mathbb{R}^{N}}\mathrm{d}\nu=1.
$$
This contradiction shows that
$$\int_{\mathbb{R}^{N}}(I_{\alpha}\ast
|u|^{\bar{p}})|u|^{\bar{p}} \mathrm{d}x=1$$
and so
$$
S_\vartheta
=
\int_{\mathbb{R}^{N}}\left(|\nabla u|^2
 -\vartheta \frac{|u|^2}{|x|^2}\right) \mathrm{d}x.
$$
It is now standard by Lagrange multiplier Theorem to get the existence of a solution to (\ref{eq1}),
and so the proof of Theorem \ref{HC} is finished.
\end{proof}

\section{Asymptotic analysis of solutions to equation (\ref{eq1})}

In this section, we study the behaviour of the solution of (\ref{eq1}) near
the origin and at the infinity.

\subsection{An equivalent problem}
In order to investigate the behaviour of the solution of (\ref{eq1}) near
the origin and at the infinity, we will transform our problem (\ref{eq1}) into
an equivalent problem in weighted space. We first give some useful lemmas.

\begin{lemma}\label{DJXLM} \ The following identity holds:
$$(N-2)\int_{\mathbb{R}^{N}}\frac{|u|^2}{|x|^2}
  \mathrm{d}x+2\int_{\mathbb{R}^{N}}\frac{u}{|x|^2}x\cdot \nabla u \mathrm{d}x=0,
  ~~~\forall\ u\in D^{1,2}(\mathbb{R}^{N}).$$
\end{lemma}
\begin{proof}[\bf Proof.]
Let $u\in C_{c}^{\infty}(\mathbb{R}^{N})$, $0\leq \beta<\frac{N-2}{2}$. We first prove that
 \begin{equation}\label{hdidplm}
 (N-2\beta-2)\int_{\mathbb{R}^{N}}\frac{|u|^2}{|x|^{2\beta+2}}\mathrm{d}x
 +2\int_{\mathbb{R}^{N}}\frac{u}{|x|^{2\beta+2}}x\cdot \nabla u \mathrm{d}x=0,
 ~~~\forall\ u\in C_{c}^{\infty}(\mathbb{R}^{N}).
\end{equation}
In fact, according to the well known divergence theorem,
we have that
 \begin{equation}\label{funcsd}
\int_{\mathbb{R}^{N}\setminus  B_{\varepsilon}(0)}\mathrm{div}\left(x\frac{|u|^2}{|x|^{2\beta+2}}\right)\mathrm{d}x
=-\frac{1}{\varepsilon^{2\beta+1}}\oint_{\partial B_{\varepsilon}(0)}u^2\mathrm{d}s.
\end{equation}
Next, we claim that
 \begin{equation}\label{fcslim}
\varliminf\limits_{\varepsilon\rightarrow 0^+}\frac{1}{\varepsilon^{2\beta+1}}\oint_
{\partial B_{\varepsilon}(0)}u^2\mathrm{d}s=0.
\end{equation}
If not, there exist $\varepsilon_0>0$, $\tau_0>0$ such that
$$
\frac{1}{\varepsilon^{2\beta+1}}\oint_{\partial B_{\varepsilon}(0)}u^2\mathrm{d}s\geq \tau_0,
~~~\forall\ 0<\varepsilon<\varepsilon_0.
$$
Since $N>2\beta+2$, we may assume that
 \begin{equation}\label{funcgjbds}
 \varepsilon_0^{N-2\beta-2}<\frac{\tau_0}{(2\beta+2)|u|_\infty^2\omega_N},
\end{equation}
where $\omega_N$ is the volume of unit ball in $\mathbb{R}^{N}$.
Then
$$
\omega_N|u|_\infty^2\varepsilon_0^N\geq
\int_{B_{\varepsilon_0}(0)}u^2\mathrm{d}x
=\int_{0}^{\varepsilon_0}\mathrm{d}\varepsilon
\oint_{\partial B_{\varepsilon}(0)}u^2\mathrm{d}s
\geq \frac{\tau_0\varepsilon_0^{2\beta+2}}{2\beta+2},
$$
which contradicts (\ref{funcgjbds}).
Therefore, there exists a sequence $\varepsilon_n\rightarrow 0^+$ such that
$$
\frac{1}{\varepsilon_n^{2\beta+1}}\oint_{\partial B_{\varepsilon_n}(0)}u^2\mathrm{d}s
\rightarrow 0~~~\mbox{as}~n\rightarrow \infty.
$$
It follows from (\ref{funcsd}) that
$$
\int_{\mathbb{R}^{N}\setminus  B_{\varepsilon_n}(0)}\mathrm{div}\left(x\frac{|u|^2}{|x|^{2\beta+2}}\right)\mathrm{d}x
\rightarrow 0~~~\mbox{as}~n\rightarrow \infty.
$$
Let $n\rightarrow \infty$, we obtain
$$
 (N-2\beta-2)\int_{\mathbb{R}^{N}}\frac{|u|^2}{|x|^{2\beta+2}}
  \mathrm{d}x+2\int_{\mathbb{R}^{N}}\frac{u}{|x|^{2\beta+2}}x\cdot \nabla u \mathrm{d}x
  =0.
$$
When $\beta=0$, we get
\begin{equation}\label{funcdsge}
(N-2)\int_{\mathbb{R}^{N}}\frac{|u|^2}{|x|^{2}}
  \mathrm{d}x+2\int_{\mathbb{R}^{N}}\frac{u}{|x|^{2}}x\cdot \nabla u \mathrm{d}x
  =0.
\end{equation}
Since $C_{c}^{\infty}(\mathbb{R}^{N})$ is dense in $ D^{1,2}(\mathbb{R}^{N})$,
we conclude that (\ref{funcdsge}) holds for any $u\in D^{1,2}(\mathbb{R}^{N})$.
\end{proof}

From Lemma \ref{DJXLM}, we get the following corollary.

\begin{corollary}\label{CIDTHD} \ If $(N-2)\beta-\beta^2=\vartheta$.
Then
$$
\int_{\mathbb{R}^{N}}\left(|\nabla u|^2
 -\vartheta \frac{|u|^2}{|x|^2}\right) \mathrm{d}x
 =\int_{\mathbb{R}^{N}}\frac{|\nabla v|^2}{|x|^{2\beta}}\mathrm{d}x,
 $$
 where $u\in D^{1,2}(\mathbb{R}^{N})$ and $v=|x|^{\beta}u$.
\end{corollary}
\begin{proof}[\bf Proof.]
By a direct calculation, one has
$$
\frac{|\nabla v|^2}{|x|^{2\beta}}=|\nabla u|^2+\beta^2\frac{|u|^2}{|x|^2}
+2\beta \frac{u}{|x|^2}x\cdot \nabla u.
$$
In view of Lemma \ref{DJXLM}, we obtain
\begin{equation*}
 \aligned
\int_{\mathbb{R}^{N}}\frac{|\nabla v|^2}{|x|^{2\beta}}\mathrm{d}x
&=\int_{\mathbb{R}^{N}}|\nabla u|^2\mathrm{d}x+\beta^2\int_{\mathbb{R}^{N}}
   \frac{|u|^2}{|x|^2}\mathrm{d}x
    +2\beta \int_{\mathbb{R}^{N}}\frac{u}{|x|^2}x\cdot \nabla u\mathrm{d}x
 \\&=\int_{\mathbb{R}^{N}}\left(|\nabla u|^2
 -((N-2)\beta-\beta^2) \frac{|u|^2}{|x|^2}\right) \mathrm{d}x
  \\&=\int_{\mathbb{R}^{N}}\left(|\nabla u|^2
 -\vartheta \frac{|u|^2}{|x|^2}\right) \mathrm{d}x.
 \endaligned
\end{equation*}
\end{proof}

We will use the following weighted function spaces.
\begin{definition}\label{WEITEDSPACE} \ Let $0<\beta<\frac{N-2}{2}$. Define a weighted Sobolev
space $D_{\beta}^{1,2}(\mathbb{R}^{N})$ as the closure of $C_{c}^{\infty}(\mathbb{R}^{N})$ with
respect to the norm
$$
\|u\|_{D_{\beta}^{1,2}(\mathbb{R}^{N})}=\left(\int_{\mathbb{R}^{N}}
\frac{|u|^{2^*}}{|x|^{2^*\beta}}\mathrm{d}x\right)^{1/{2^*}}+
\left(\int_{\mathbb{R}^{N}}
\frac{|\nabla u|^2}{|x|^{2\beta}}\mathrm{d}x\right)^{1/2}.
$$
Let $\Omega$ be an open subset of $\mathbb{R}^{N}$, we define
$$
W_{\beta}^{1,2}(\Omega):= \left\{u:\Omega\rightarrow \mathbb{R}~ \mbox{measurable}:
\frac{u}{|x|^{\beta}}\in L^{2^*}(\Omega),~\frac{|\nabla u|}{|x|^{\beta}}\in L^{2}(\Omega)\right\}
$$
with norm
$$
\|u\|_{W_{\beta}^{1,2}(\Omega)}=\left(\int_{\Omega}
\frac{|u|^{2^*}}{|x|^{2^*\beta}}\mathrm{d}x\right)^{1/{2^*}}+
\left(\int_{\Omega}
\frac{|\nabla u|^2}{|x|^{2\beta}}\mathrm{d}x\right)^{1/2}.
$$
Note that, following a standard argument \rm{(e.g.,\cite{DSHUKH})}, one has that the space $D_{\beta}^{1,2}(\mathbb{R}^{N})$
coincides with $W_{\beta}^{1,2}(\mathbb{R}^{N})$.
\end{definition}

\begin{lemma}\label{Lgjgbds} \ Let $0<\beta<\frac{N-2}{2}$, $u\in D_{\beta}^{1,2}(\mathbb{R}^{N})$.
Then
$$
\int_{\mathbb{R}^{N}}\frac{|u|^2}{|x|^{2\beta+2}}\mathrm{d}x
\leq
\left(\frac{2}{N-2\beta-2}\right)^2
\int_{\mathbb{R}^{N}}\frac{|\nabla u|^2}{|x|^{2\beta}}\mathrm{d}x.
$$
\end{lemma}
\begin{proof}[\bf Proof.]
We deduce from (\ref{hdidplm}) and the H\"{o}lder inequality that
$$
 \int_{\mathbb{R}^{N}}\frac{|u|^2}{|x|^{2\beta+2}}\mathrm{d}x
 \leq
 \frac{2}{N-2\beta-2}\left(\int_{\mathbb{R}^{N}}\frac{|u|^2}{|x|^{2\beta+2}}\mathrm{d}x\right)^{1/2}
 \left(\int_{\mathbb{R}^{N}}\frac{|\nabla u|^2}{|x|^{2\beta}}\mathrm{d}x\right)^{1/2},
$$
which leads to
 \begin{equation}\label{cmxine}
\int_{\mathbb{R}^{N}}\frac{|u|^2}{|x|^{2\beta+2}}\mathrm{d}x
\leq
\left(\frac{2}{N-2\beta-2}\right)^2
\int_{\mathbb{R}^{N}}\frac{|\nabla u|^2}{|x|^{2\beta}}\mathrm{d}x,
~~~\forall\ u\in C_{c}^{\infty}(\mathbb{R}^{N}).
\end{equation}
Since $C_{c}^{\infty}(\mathbb{R}^{N})$ is dense in $ D_{\beta}^{1,2}(\mathbb{R}^{N})$,
we conclude that (\ref{cmxine}) holds for any $u\in D_{\beta}^{1,2}(\mathbb{R}^{N})$.
\end{proof}

\begin{remark}\label{djxkjql}
From Lemma \ref{Lgjgbds}, it is easy to see that
$$
u\in D_{\beta}^{1,2}(\mathbb{R}^{N})~~~\iff~~~\frac{u}{|x|^\beta}\in D^{1,2}(\mathbb{R}^{N}).
$$
\end{remark}

\begin{lemma}\label{SOBOINETG} \ Let $0<\beta<\frac{N-2}{2}$. Then there is a constant
$C(N,\beta)>0$ such that
$$
\left(\int_{\mathbb{R}^{N}}
\frac{|v|^{2^*}}{|x|^{2^*\beta}}\mathrm{d}x\right)^{1/{2^*}}
\leq
C(N,\beta)
\left(\int_{\mathbb{R}^{N}}
\frac{|\nabla v|^2}{|x|^{2\beta}}\mathrm{d}x\right)^{1/2},~~~
\forall~v\in D_{\beta}^{1,2}(\mathbb{R}^{N}).
$$
\end{lemma}
\begin{proof}[\bf Proof.]
From Remark \ref{djxkjql}, $~\frac{v}{|x|^\beta}\in D^{1,2}(\mathbb{R}^{N})$. The Sobolev
inequality leads to
$$
\left(\int_{\mathbb{R}^{N}}
\frac{|v|^{2^*}}{|x|^{2^*\beta}}\mathrm{d}x\right)^{1/{2^*}}
\leq
S(N)
\left(\int_{\mathbb{R}^{N}}
\left|\nabla \left(\frac{v}{|x|^{\beta}}\right)\right|^2\mathrm{d}x\right)^{1/2},
$$
where $S(N)$ is the best Sobolev constant for the embedding
$D^{1,2}(\mathbb{R}^{N})\hookrightarrow L^{2^*}(\mathbb{R}^{N})$.
By Lemma \ref{Lgjgbds}, we have
\begin{equation*}
 \aligned
\int_{\mathbb{R}^{N}}
\left|\nabla \left(\frac{v}{|x|^{\beta}}\right)\right|^2\mathrm{d}x
&\leq
2\int_{\mathbb{R}^{N}}
 \frac{|\nabla v|^2}{|x|^{2\beta}}\mathrm{d}x
 +2\beta^2\int_{\mathbb{R}^{N}}
 \frac{| v|^2}{|x|^{2\beta+2}}\mathrm{d}x
 \\&\leq
 \left(2+2\beta^2\left(\frac{2}{N-2\beta-2}\right)^2\right)
 \int_{\mathbb{R}^{N}}
\frac{|\nabla v|^2}{|x|^{2\beta}}\mathrm{d}x.
 \endaligned
\end{equation*}
The proof is thus finished.
\end{proof}

\begin{lemma}\label{Lemfcdjbh} \ Let $0<\vartheta<\frac{(N-2)^2}{4}$, $0<\beta<\frac{N-2}{2}$
and $(N-2)\beta-\beta^2=\vartheta$. Suppose that $u\in D^{1,2}(\mathbb{R}^{N})$ is
a weak solution of {\rm (\ref{eq1})}. Set $v=|x|^\beta u$, then
$$
\int_{\mathbb{R}^{N}}\frac{\nabla v \cdot \nabla \varphi }{|x|^{2\beta}}\mathrm{d}x
=\int_{\mathbb{R}^{N}}\left(I_{\alpha}\ast\frac{|v|^{\bar{p}}}{|x|^{\bar{p}\beta}}\right)
  \frac{|v|^{\bar{p}-2}v\varphi}{|x|^{\bar{p}\beta}}\mathrm{d}x,~~~\forall  ~\varphi \in
  C_{c}^{\infty}(\mathbb{R}^{N}).
$$
\end{lemma}

\begin{proof}[\bf Proof.]
Testing (\ref{eq1}) with $\frac{\varphi}{|x|^{\beta}}$, we obtain
$$
\int_{\mathbb{R}^{N}}\frac{\nabla u \cdot \nabla \varphi }{|x|^{\beta}}\mathrm{d}x
-\beta \int_{\mathbb{R}^{N}}\frac{\varphi x\cdot\nabla u }{|x|^{\beta+2}}\mathrm{d}x
-\vartheta \int_{\mathbb{R}^{N}}\frac{\varphi u }{|x|^{\beta+2}}\mathrm{d}x
=\int_{\mathbb{R}^{N}}(I_{\alpha}\ast |u|^{\bar{p}})
  \frac{|u|^{\bar{p}-2}u\varphi}{|x|^{\beta}}\mathrm{d}x.
$$
Note that $(N-2)\beta-\beta^2=\vartheta$, we have
\begin{equation*}
  \aligned
\int_{\mathbb{R}^{N}}\frac{\nabla v \cdot \nabla \varphi }{|x|^{2\beta}}\mathrm{d}x
&=\int_{\mathbb{R}^{N}}\frac{\nabla u \cdot \nabla \varphi }{|x|^{\beta}}\mathrm{d}x
   +\beta \int_{\mathbb{R}^{N}}\frac{u x\cdot\nabla\varphi }{|x|^{\beta+2}}\mathrm{d}x
\\&=\beta \int_{\mathbb{R}^{N}}\frac{\varphi x\cdot\nabla u }{|x|^{\beta+2}}\mathrm{d}x
      +\vartheta \int_{\mathbb{R}^{N}}\frac{\varphi u }{|x|^{\beta+2}}\mathrm{d}x
      \\&~~~~~+\beta \int_{\mathbb{R}^{N}}\frac{u x\cdot\nabla\varphi }{|x|^{\beta+2}}\mathrm{d}x
      +\int_{\mathbb{R}^{N}}\left(I_{\alpha}\ast\frac{|v|^{\bar{p}}}{|x|^{\bar{p}\beta}}\right)
  \frac{|v|^{\bar{p}-2}v\varphi}{|x|^{\bar{p}\beta}}\mathrm{d}x
  \\&=\beta\int_{\mathbb{R}^{N}}\mathrm{div}\left(\frac{\varphi u}{|x|^{\beta+2}}x\right)\mathrm{d}x
        +\int_{\mathbb{R}^{N}}\left(I_{\alpha}\ast\frac{|v|^{\bar{p}}}{|x|^{\bar{p}\beta}}\right)
  \frac{|v|^{\bar{p}-2}v\varphi}{|x|^{\bar{p}\beta}}\mathrm{d}x.
   \endaligned
\end{equation*}
By the divergence theorem, one has
\begin{equation*}
\int_{\mathbb{R}^{N}\setminus  B_{\varepsilon}(0)}\mathrm{div}\left(\frac{\varphi u}{|x|^{\beta+2}}x\right)\mathrm{d}x
=-\frac{1}{\varepsilon^{\beta+1}}\oint_{\partial B_{\varepsilon}(0)}\varphi u\mathrm{d}s.
\end{equation*}
Since $u\in D^{1,2}(\mathbb{R}^{N})$, by the H\"{o}lder inequality, we deduce that
$$
\int_{B_{\varepsilon}(0)}|u|\mathrm{d}x\leq
O(\varepsilon^{\frac{N}{2}+1}).
$$
Note that $N>2\beta+2$. Now, using the same arguments as in the proof of Lemma \ref{DJXLM}, we can get that
$$
\varliminf\limits_{\varepsilon\rightarrow 0^+}\frac{1}{\varepsilon^{\beta+1}}\oint_
{\partial B_{\varepsilon}(0)}\varphi u\mathrm{d}s=0.
$$
Hence,
$$
\int_{\mathbb{R}^{N}}\mathrm{div}\left(\frac{\varphi u}{|x|^{\beta+2}}x\right)\mathrm{d}x=0,
$$
and
$$
\int_{\mathbb{R}^{N}}\frac{\nabla v \cdot \nabla \varphi }{|x|^{2\beta}}\mathrm{d}x
=\int_{\mathbb{R}^{N}}\left(I_{\alpha}\ast\frac{|v|^{\bar{p}}}{|x|^{\bar{p}\beta}}\right)
  \frac{|v|^{\bar{p}-2}v\varphi}{|x|^{\bar{p}\beta}}\mathrm{d}x.
$$
\end{proof}

Inspired by Lemma \ref{Lemfcdjbh}, we give
\begin{definition}\label{TRSINEWS} \ Let $f\in L^{1}_{loc}(\mathbb{R}^{N})$. Then a function $v\in D_{\beta}^{1,2}(\mathbb{R}^{N})$
will be called a weak solution of
the equation
$$
-\Delta_{\beta} v=f
~~~{\rm in}~\mathbb{R}^{N},
$$
if
$$
\int_{\mathbb{R}^{N}}\frac{\nabla v \cdot \nabla \varphi }{|x|^{2\beta}}\mathrm{d}x
=\int_{\mathbb{R}^{N}}f \varphi\mathrm{d}x~~~{\rm for~ all}~\varphi\in C_{c}^{\infty}(\mathbb{R}^{N}).
$$
\end{definition}

\begin{remark}\label{EQTRREMR}
In view of Lemma \ref{Lemfcdjbh}, we see that if $u\in D^{1,2}(\mathbb{R}^{N})$ is
a weak solution to
$$
-\Delta u-\frac{\vartheta}{|x|^2}u=(I_{\alpha}\ast|u|^{\bar{p}})|u|^{\bar{p}-2}u
~~~{\rm in}~\mathbb{R}^{N},
$$
then $v=|x|^\beta u\in D_{\beta}^{1,2}(\mathbb{R}^{N})$ weakly solves the equation
$$
-\Delta_{\beta} v
=\left(I_{\alpha}\ast\frac{|v|^{\bar{p}}}{|x|^{\bar{p}\beta}}\right)
  \frac{|v|^{\bar{p}-2}v}{|x|^{\bar{p}\beta}}
~~~{\rm in}~\mathbb{R}^{N},
$$
where $\beta=\frac{N-2-\sqrt{(N-2)^2-4\vartheta}}{2}$.
\end{remark}

\subsection{The $L^{\infty}$ estimate}
In this section we prove a regularity result for the weak solution of (\ref{eq1}).
More precisely:

\begin{lemma}\label{Linfesti} \ Let $(N-4)_{+}<\alpha<N$, $0<\beta<\frac{N-2}{2}$. Suppose that
$0\leq v\in D_{\beta}^{1,2}(\mathbb{R}^{N})$ be a weak solution of
\begin{equation} \label{chodjiadfjubeq}
-\Delta_{\beta} v
=\left(I_{\alpha}\ast\frac{v^{\bar{p}}}{|x|^{\bar{p}\beta}}\right)
  \frac{v^{\bar{p}-1}}{|x|^{\bar{p}\beta}}
~~~{\rm in}~\mathbb{R}^{N}.
\end{equation}
Then $v\in L^{\infty}(\mathbb{R}^{N})$.
\end{lemma}
\begin{proof}[\bf Proof.]
We aim to apply Moser's iteration technique (e.g., \cite{LEOPERPRI}) to prove this lemma.
Let us define for $\lambda>1$ and $T>0$,
$$
\phi(t)=\left\{
     \begin{aligned}
      & t^\lambda,&0\leq t\leq T, \\
      & \lambda T^{\lambda-1}t-(\lambda-1)T^{\lambda},&t>T. \\
      \end{aligned}
      \right.
$$
Then $\phi(v)\in D_{\beta}^{1,2}(\mathbb{R}^{N})$. Since $\phi'(v)v\leq \lambda\phi(v)$,
by Lemma \ref{SOBOINETG}, one has
\begin{equation} \label{sdsgjbds}
\aligned
\left(\int_{\mathbb{R}^{N}}
\frac{|\phi(v)|^{2^*}}{|x|^{2^*\beta}}\mathrm{d}x\right)^{2/{2^*}}
&\leq
C(N,\beta)^2
\int_{\mathbb{R}^{N}}
\frac{|\nabla \phi(v)|^2}{|x|^{2\beta}}\mathrm{d}x
\\&\leq C(N,\beta)^2
\int_{\mathbb{R}^{N}}
\frac{\nabla (\phi(v)\phi'(v))\cdot\nabla v}{|x|^{2\beta}}\mathrm{d}x
\\&=C(N,\beta)^2
\int_{\mathbb{R}^{N}}
\left(I_{\alpha}\ast\frac{v^{\bar{p}}}{|x|^{\bar{p}\beta}}\right)
  \frac{v^{\bar{p}-1}\phi(v)\phi'(v)}{|x|^{\bar{p}\beta}}\mathrm{d}x
\\&\leq \lambda  C(N,\beta)^2
\int_{\mathbb{R}^{N}}
\left(I_{\alpha}\ast\frac{v^{\bar{p}}}{|x|^{\bar{p}\beta}}\right)
  \frac{v^{\bar{p}-2}\phi^2(v)}{|x|^{\bar{p}\beta}}\mathrm{d}x.
\endaligned
\end{equation}
Let $\lambda=2^*/2$, $m>0$. By the Hardy-Littlewood-Sobolev inequality, we have
\begin{align*}
&\left(\int_{\mathbb{R}^{N}}
\frac{|\phi(v)|^{2^*}}{|x|^{2^*\beta}}\mathrm{d}x\right)^{2/{2^*}}
\\&~\leq
  \frac{2^*}{2}  C(N,\beta)^2
     \int_{\{v\leq m\}}
      \left(I_{\alpha}\ast\frac{v^{\bar{p}}}{|x|^{\bar{p}\beta}}\right)
       \frac{v^{\bar{p}-2}\phi^2(v)}{|x|^{\bar{p}\beta}}\mathrm{d}x
\\&~~~+\frac{2^*}{2}  C(N,\beta)^2
     \int_{\{v> m\}}
      \left(I_{\alpha}\ast\frac{v^{\bar{p}}}{|x|^{\bar{p}\beta}}\right)
       \frac{v^{\bar{p}-2}\phi^2(v)}{|x|^{\bar{p}\beta}}\mathrm{d}x
\\&~\leq
    \frac{2^*}{2}  C(N,\beta)^2
     \int_{\{v\leq m\}}
      \left(I_{\alpha}\ast\frac{v^{\bar{p}}}{|x|^{\bar{p}\beta}}\right)
       \frac{v^{\bar{p}+2^*-2}}{|x|^{\bar{p}\beta}}\mathrm{d}x
\\&~~~+\frac{2^*}{2}  C(N,\beta)^2
     \left|I_{\alpha}\ast\frac{v^{\bar{p}}}{|x|^{\bar{p}\beta}}\right|_{\frac{2^*}{2^*-\bar{p}}}
      \left(\int_{\{v> m\}}\frac{v^{2^*}}{|x|^{2^*\beta}}\mathrm{d}x\right)
        ^{\frac{\bar{p}-2}{2^*}}
        \left(\int_{\mathbb{R}^{N}}\frac{|\phi(v)|^{2^*}}{|x|^{2^*\beta}}\mathrm{d}x\right)
        ^{\frac{2}{2^*}}
\\&~\leq
   \frac{2^*}{2}  C(N,\beta)^2 S(N,\alpha)
     \left(\int_{\mathbb{R}^{N}}\frac{v^{2^*}}{|x|^{2^*\beta}}\mathrm{d}x\right)
        ^{\frac{\bar{p}}{2^*}}
\\&~~~\times\left(m^{2^*-2}\left(\int_{\mathbb{R}^{N}}\frac{v^{2^*}}{|x|^{2^*\beta}}\mathrm{d}x\right)
        ^{\frac{\bar{p}}{2^*}}+\left(\int_{\{v> m\}}\frac{v^{2^*}}{|x|^{2^*\beta}}\mathrm{d}x\right)
        ^{\frac{\bar{p}-2}{2^*}}
        \left(\int_{\mathbb{R}^{N}}\frac{|\phi(v)|^{2^*}}{|x|^{2^*\beta}}\mathrm{d}x\right)
        ^{\frac{2}{2^*}}\right).
\end{align*}
Choosing $m>0$ such that
$$
\frac{2^*}{2}  C(N,\beta)^2 S(N,\alpha)
        \left(\int_{\mathbb{R}^{N}}\frac{v^{2^*}}{|x|^{2^*\beta}}\mathrm{d}x\right)
        ^{\frac{\bar{p}}{2^*}}
          \left(\int_{\{v> m\}}\frac{v^{2^*}}{|x|^{2^*\beta}}\mathrm{d}x\right)
        ^{\frac{\bar{p}-2}{2^*}}<\frac{1}{2}.
$$
Thus, we obtain
$$
\left(\int_{\mathbb{R}^{N}}
\frac{|\phi(v)|^{2^*}}{|x|^{2^*\beta}}\mathrm{d}x\right)^{2/{2^*}}
\leq
2^* C(N,\beta)^2  S(N,\alpha)m^{2^*-2}
     \left(\int_{\mathbb{R}^{N}}\frac{v^{2^*}}{|x|^{2^*\beta}}\mathrm{d}x\right)
        ^{\frac{2\bar{p}}{2^*}}.
$$
Let $T\rightarrow \infty$, we get
$$
\int_{\mathbb{R}^{N}}\frac{v^\frac{(2^*)^2}{2}}{|x|^{2^*\beta}}\mathrm{d}x<\infty.
$$
Next, we use Moser's iteration. For $k\geq 1$, let us define $\{\lambda_k\}$ as
$$
\lambda_{k+1}-1=\frac{\bar{p}}{2}(\lambda_{k}-1)~~~{\rm and}~~~\lambda_1=2^*/2.
$$
In view of (\ref{sdsgjbds}), let $\lambda=\lambda_{k+1}$, one has
\begin{align*}
\frac{1}{C(N,\beta)^2}\left(\int_{\mathbb{R}^{N}}
\frac{|\phi(v)|^{2^*}}{|x|^{2^*\beta}}\mathrm{d}x\right)^{2/{2^*}}
&\leq
\lambda_{k+1}
\int_{\mathbb{R}^{N}}
\left(I_{\alpha}\ast\frac{v^{\bar{p}}}{|x|^{\bar{p}\beta}}\right)
  \frac{v^{\bar{p}-2}\phi^2(v)}{|x|^{\bar{p}\beta}}\mathrm{d}x
\\&\leq
\lambda_{k+1}
\int_{\mathbb{R}^{N}}
\left(I_{\alpha}\ast\frac{v^{\bar{p}}}{|x|^{\bar{p}\beta}}\right)
  \frac{v^{\bar{p}-2+2\lambda_{k+1}}}{|x|^{\bar{p}\beta}}\mathrm{d}x
\\&\leq
\lambda_{k+1} S(N,\alpha)
 \left(\int_{\mathbb{R}^{N}}\frac{v^{2^*}}{|x|^{2^*\beta}}\mathrm{d}x\right)
 ^{\frac{\bar{p}}{2^*}}
 \left(\int_{\mathbb{R}^{N}}\frac{v^{2^*\lambda_k}}{|x|^{2^*\beta}}\mathrm{d}x\right)
 ^{\frac{\bar{p}}{2^*}}.
\end{align*}
Let $T\rightarrow \infty$, we have that
$$\left(\int_{\mathbb{R}^{N}}\frac{v^{2^*\lambda_{k+1}}}{|x|^{2^*\beta}}\mathrm{d}x\right)
 ^{\frac{1}{2^*(\lambda_{k+1}-1)}}
 \leq
 C_{k+1}
  \left(\int_{\mathbb{R}^{N}}\frac{v^{2^*\lambda_k}}{|x|^{2^*\beta}}\mathrm{d}x\right)
 ^{\frac{1}{2^*(\lambda_{k}-1)}},
$$
where
$$
C_{k+1}:=\left(C(N,\beta)^2 S(N,\alpha)\lambda_{k+1}
 \left(\int_{\mathbb{R}^{N}}\frac{v^{2^*}}{|x|^{2^*\beta}}\mathrm{d}x\right)
 ^{\frac{\bar{p}}{2^*}}\right)^{\frac{1}{2(\lambda_{k+1}-1)}}.
$$
Write
$$A_k= \left(\int_{\mathbb{R}^{N}}\frac{v^{2^*\lambda_k}}{|x|^{2^*\beta}}\mathrm{d}x\right)
 ^{\frac{1}{2^*(\lambda_{k}-1)}}$$
for $k\geq 1$. Therefore,
$$A_{k+1}\leq C_{k+1}A_{k}.$$
Note that
$$
\ln A_{k+1}\leq
\sum_{j\geq 2}\ln C_j+\ln A_1:=\bar{C}<+\infty.
$$
For $R>1$ fixed,
$$
\bar{C}
\geq {\frac{\lambda_{k+1}}{\lambda_{k+1}-1}}\ln \left(\int_{B_{R}(0)}
v^{2^*\lambda_{k+1}}\mathrm{d}x\right)^{\frac{1}{2^*\lambda_{k+1}}}
-\frac{\beta\ln R}{\lambda_{k+1}-1}.
$$
Hence, for $k>0$ sufficiently large,
$$
 \left(\int_{B_{R}(0)}
v^{2^*\lambda_{k+1}}\mathrm{d}x\right)^{\frac{1}{2^*\lambda_{k+1}}}
\leq
e^ {2\bar{C}}.
$$
Since $\lambda_k\rightarrow \infty$ as $k\rightarrow \infty$, we have that
$|v|_{\infty,B_R(0)}\leq e^ {2\bar{C}}$, with $\bar{C}$ a positive constant
not depending on $R$. We deduce that
$$|v|_{\infty,\mathbb{R}^{N}}\leq e^ {2\bar{C}}.$$
The proof is thus finished.
\end{proof}

\subsection{Weak Harnack inequality}
In this section, we are going to establish a weak Harnark inequality for nonlocal operator
$-\Delta_{\beta}$. In order to achieve it, we need the following weighted Poincar\'{e}-Sobolev
inequality. For simplicity, we set
$$\mathrm{d}\gamma=\frac{\mathrm{d}x}{|x|^{2\beta}}~~~{\rm and}~~~\mathrm{d}\tau=\frac{\mathrm{d}x}{|x|^{2^*\beta}}.$$
\begin{proposition}\label{posoin} \ Let $0<\beta<\frac{N-2}{2}$. For any $\varepsilon>0$, there exists positive
constant $C=C(N,\beta,\varepsilon)$ such that for $u\in W_{\beta}^{1,2}(B_{1}(z))$ with
$$
|\left\{x\in B_{1}(z):u(x)=0\right\}|_{\mathrm{d}\gamma}
\geq
\varepsilon
| B_{1}(z)|_{\mathrm{d}\gamma},
$$
there holds
$$
\int_{B_{1}(z)} u^2\mathrm{d}\gamma
\leq
C
\int_{B_{1}(z)}|\nabla u|^2\mathrm{d}\gamma.
$$
\end{proposition}
Let us give some preliminary lemmas before the proof. We first recall the definition of
Muckenhoupt $A_p$ condition \cite{Muckenhoupt}. A Lebesgue measure function
$w:\mathbb{R}^{N}\rightarrow [0,+\infty)$ is said to belong to the class $A_p(\mathbb{R}^{N})$
for $1<p<+\infty$ if there is a positive number $C(N,p)$ such that
$$\left(\frac{1}{|B|}\int_{B} w(x)\mathrm{d}x\right)
\left(\frac{1}{|B|}\int_{B} w(x)^{\frac{-1}{p-1}}\mathrm{d}x\right)^{p-1}
\leq
C(N,p)
$$
for any ball $B\subset \mathbb{R}^{N}$. In view of \cite{Kilpel}, we can see that the
functions $\frac{1}{|x|^{2\beta}}$ and $\frac{1}{|x|^{2^*\beta}}$ for $0<\beta<\frac{N-2}{2}$
belong to the class $A_2(\mathbb{R}^{N})$. A measure $\mu$ is said to be doubling if there exists
a constant $C\geq 1$ such that
$$
|B_{2\rho}(z)|_{\mathrm{d}\mu}\leq C |B_{\rho}(z)|_{\mathrm{d}\mu}
$$
for all $z\in \mathbb{R}^{N}$ and $\rho>0$.

\begin{remark}\label{MUDOUTJ}
By {\rm\cite{Stein}} the Muckenhoupt $A_p$ condition leads to the doubling condition for the weighted measure.
Note that $2\beta<2^*\beta<N$. Hence, the measures $\gamma$ and $\tau$ must be doubling. The constant
$$
\bar{C}(N,\beta)=\sup_{B_{\rho}(z)\subset  \mathbb{R}^{N}}
\frac{|B_{2\rho}(z)|_{\mathrm{d}\gamma}}{|B_{\rho}(z)|_{\mathrm{d}\gamma}}
$$
is called the doubling constant of $\gamma$.
\end{remark}

Set $|B_R|_{\mathrm{d}w}=\int_{B_R} w(x)\mathrm{d}x$. The following Poincar\'{e}
inequality is valid.
\begin{lemma}\label{AMPOINC} {\rm(\cite{FABESKESE})} \ If $w\in A_2(\mathbb{R}^{N})$. Then there
exist positive constants $C$ and $\varepsilon$ such that for all balls $B_R$
with radii $R$ and all numbers $k$ satisfying $1\leq k \leq \frac{N}{N-1}+\varepsilon$,
$$
\left(\frac{1}{|B_R|_{\mathrm{d}w}}\int_{B_R}
|u(x)-\bar{u}_{B_R}|^{2k}w(x)\mathrm{d}x\right)^{1/2k}
\leq
CR
\left(\frac{1}{|B_R|_{\mathrm{d}w}}\int_{B_R}
|\nabla u|^2w(x)\mathrm{d}x\right)^{1/2}
$$
for any $u\in C^{1}(\overline{B_R})$,
where either $\bar{u}_{B_R}=\frac{1}{|B_R|_{\mathrm{d}w}}\int_{B_R}u(x)w(x)\mathrm{d}x$
or $\bar{u}_{B_R}=\frac{1}{|B_R|}\int_{B_R}u(x)\mathrm{d}x$.
\end{lemma}

We also need the following Poincar\'{e}-Sobolev inequality.

\begin{lemma}\label{CLAPOINSOBOIN} {\rm(\cite{HANLIN})} \ For any $\varepsilon>0$, there exists positive
constant
$C=C(N, \varepsilon)$ such that for $u\in  H^{1}(B_{1})$ with
$$
|\left\{x\in B_{1}:u(x)=0\right\}|
\geq
\varepsilon
| B_{1}|,
$$
there holds
$$
\int_{B_{1}} u^2\mathrm{d}x
\leq
C
\int_{B_{1}}|\nabla u|^2\mathrm{d}x.
$$
\end{lemma}

We are now ready to prove Proposition \ref{posoin}.
\begin{proof}[\bf Proof of Proposition \ref{posoin}.] We distinguish two cases.
\par\noindent
Case 1: $|z|\geq 3$. Then
\begin{align*}
|\left\{x\in B_{1}(z):u(x)=0\right\}|
&\geq
(|z|-1)^{2\beta}|\left\{x\in B_{1}(z):u(x)=0\right\}|_{\mathrm{d}\gamma}
\\&\geq
    (|z|-1)^{2\beta}\varepsilon| B_{1}(z)|_{\mathrm{d}\gamma}
\geq
\left(\frac{|z|-1}{|z|+1}\right)^{2\beta}\varepsilon| B_{1}(z)|
\\&\geq \frac{\varepsilon}{4^\beta}|B_{1}(z)|.
\end{align*}
By Lemma \ref{CLAPOINSOBOIN}, we see
\begin{align*}
\int_{B_{1}(z)} u^2\mathrm{d}\gamma
&\leq
\frac{1}{(|z|-1)^{2\beta}}
\int_{B_{1}(z)} u^2\mathrm{d}x
\leq
\frac{C(N,\beta,\varepsilon)}{(|z|-1)^{2\beta}}\int_{B_{1}(z)}|\nabla u|^2\mathrm{d}x
\\&
\leq
   \left(\frac{|z|+1}{|z|-1}\right)^{2\beta}
   C(N,\beta,\varepsilon)\int_{B_{1}(z)}|\nabla u|^2\mathrm{d}\gamma
\leq C(N,\beta,\varepsilon)\int_{B_{1}(z)}|\nabla u|^2\mathrm{d}\gamma.
\end{align*}
Case 2: $|z|< 3$. Arguing indirectly, suppose that there exists a sequence
$\{u_m\} \subset W_{\beta}^{1,2}(B_{1}(z_m))$ such that $|z_m|< 3$,
\begin{equation} \label{supepga}
|\left\{x\in B_{1}(z_m):u_m(x)=0\right\}|_{\mathrm{d}\gamma}
\geq
\varepsilon
| B_{1}(z_m)|_{\mathrm{d}\gamma},
\end{equation}
and
\begin{equation} \label{ysjxad}
\int_{B_{1}(z_m)} u_m^2\mathrm{d}\gamma=1,~~~
\int_{B_{1}(z_m)}|\nabla u_m|^2\mathrm{d}\gamma\rightarrow 0.
\end{equation}
Let $z_m\rightarrow z_0$ in $\mathbb{R}^{N}$. Since
$|x|^{-2\beta}\in L^{1}_{loc}(\mathbb{R}^{N})$, there exists $0<\lambda<1$
such that for $m$ sufficiently large
$B_{\lambda}(z_0)\subset B_{1}(z_m)$
and
\begin{equation} \label{oestiguji}
|B_{1}(z_m)\setminus B_{\lambda}(z_0)|_{\mathrm{d}\gamma}<
\min \left\{\frac{|B_1|}{2^{3N+4\beta}},~\frac{\varepsilon|B_1|}{2^{4\beta+1}}\right\}.
\end{equation}
Since $|z_m|< 3$, $|B_{1}(z_m)|_{\mathrm{d}\gamma}\geq 16^{-\beta}|B_{1}|$.
It follows from (\ref{ysjxad}), Lemma \ref{AMPOINC} and the H\"{o}lder inequality that
\begin{equation} \label{cityyneggji}
\aligned
&\left(\int_{B_{1}(z_m)}
|u_m|^{\frac{2N}{N-1}}
\mathrm{d}\gamma\right)^{{\frac{N-1}{2N}}}
\\&~~~~~\leq
   \left(\int_{B_{1}(z_m)}
     |u_m-\bar{u}_m|^{\frac{2N}{N-1}}
     \mathrm{d}\gamma\right)^{{\frac{N-1}{2N}}}
       +\left(\int_{B_{1}(z_m)}
         |\bar{u}_m|^{\frac{2N}{N-1}}
           \mathrm{d}\gamma\right)^{{\frac{N-1}{2N}}}
\\&~~~~~\leq |B_{1}(z_m)|_{\mathrm{d}\gamma}^{-\frac{1}{2N}}
         \left(C\left(\int_{B_{1}(z_m)}|\nabla u_m|^2\mathrm{d}\gamma\right)^{\frac{1}{2}}
         + \left(\int_{B_{1}(z_m)}u_m^2\mathrm{d}\gamma\right)^{\frac{1}{2}}\right)
\\&~~~~~\leq 2|B_{1}(z_m)|_{\mathrm{d}\gamma}^{-\frac{1}{2N}}
\\&~~~~~\leq 2 (16^{-\beta}|B_{1}|)^{-\frac{1}{2N}},
\endaligned
\end{equation}
where $\bar{u}_m=\frac{1}{|B_{1}(z_m)|_{\mathrm{d}\gamma}}
\int_{B_{1}(z_m)}u_m\mathrm{d}\gamma$.
Then from (\ref{oestiguji}) and the H\"{o}lder inequality, we deduce
\begin{align*}
\int_{B_{1}(z_m)\setminus B_{\lambda}(z_0)} u_m^2\mathrm{d}\gamma
&\leq
   \left(\int_{B_{1}(z_m)\setminus B_{\lambda}(z_0)} |u_m|
   ^{\frac{2N}{N-1}}\mathrm{d}\gamma\right)^{{\frac{N-1}{N}}}
   |B_{1}(z_m)\setminus B_{\lambda}(z_0)|_{\mathrm{d}\gamma}^{{\frac{1}{N}}}
\\&
\leq \frac{1}{2}.
\end{align*}
Therefore,
\begin{equation} \label{cguolyibuf}
\int_{B_{\lambda}(z_0)} u_m^2\mathrm{d}\gamma
\geq \frac{1}{2}.
\end{equation}
Since $\{u_m\}$ and $\{\frac{u_m}{|x|^{\beta-1}}\}$ are
bounded in $ H^1(B_{\lambda}(z_0))$, we may assume that
$$u_m\rightharpoonup u_0~~~\mbox{in}~H^1(B_{\lambda}(z_0)),$$
$$\frac{u_m}{|x|^{\beta-1}} \rightharpoonup \frac{u_0}{|x|^{\beta-1}}~~~\mbox{in}~H^1(B_{\lambda}(z_0)),$$
\begin{equation} \label{jqyigslian}
\frac{u_m}{|x|^{\beta-1}} \rightarrow \frac{u_0}{|x|^{\beta-1}}~~~\mbox{in}~ L^{s}(B_{\lambda}(z_0))~(1\leq s <2^*),
\end{equation}
\begin{equation} \label{dsyigesl}
\nabla u_m \rightharpoonup \nabla u_0~~~\mbox{in}~ \left(L^{2}(B_{\lambda}(z_0))\right)^N.
\end{equation}
By (\ref{cityyneggji}), we see that $\{\int_{B_{\lambda}(z_0)} |u_m-u_0|^{\frac{2N}{N-1}}\mathrm{d}\gamma\}$
is bounded. Set $s=\frac{6N^2-4N}{3N^2-5N-6\beta-2}$. Then $1<s<2^*$ since $0<\beta<\frac{N-2}{2}$.
Hence, (\ref{jqyigslian}) leads to
\begin{equation} \label{lishinfmin}
\aligned
\int_{B_{\lambda}(z_0)} |u_m-u_0|^2\mathrm{d}\gamma
&\leq \left(\int_{B_{\lambda}(z_0)} \frac{|u_m-u_0|^s}{|x|^{(\beta-1)s}}\mathrm{d}x\right)^{\frac{1}{s}}
\left(\int_{B_{\lambda}(z_0)} |u_m-u_0|^{\frac{2N}{N-1}}\mathrm{d}\gamma\right)^{\frac{N-1}{2N}}
\\&~~~~~~\times\left(\int_{B_{\lambda}(z_0)} \frac{\mathrm{d}x}{|x|^{N-\frac{2}{3}}}\right)^{\frac{3(N+\beta)}{3N^2-2N}}
\\&=o(1).
\endaligned
\end{equation}
From (\ref{ysjxad}) and (\ref{dsyigesl}), we have that
$$\int_{B_{\lambda}(z_0)}|\nabla u_0|^2\mathrm{d}x=0,$$
which implies
$$\nabla u_0=0~~~\mbox{on}~B_{\lambda}(z_0).$$
It follows from (\ref{cguolyibuf}) and (\ref{lishinfmin}) that
$$\int_{B_{\lambda}(z_0)} u_0^2\mathrm{d}\gamma
\geq \frac{1}{2}.$$
Thus $u_0\equiv C \neq 0~~\mbox{on}~B_{\lambda}(z_0)$. By (\ref{oestiguji}), we have
$$|B_{1}(z_m)\setminus B_{\lambda}(z_0)|_{\mathrm{d}\gamma}\leq
\frac{\varepsilon}{2}|B_{1}(z_m)|_{\mathrm{d}\gamma},
$$
which, together with (\ref{supepga}), yields
$$|\left\{x\in B_{\lambda}(z_0):u_m(x)=0\right\}|_{\mathrm{d}\gamma}
\geq \frac{\varepsilon}{2}|B_{1}(z_m)|_{\mathrm{d}\gamma}.
$$
Hence,
\begin{align*}
0&=\lim\limits_{m\rightarrow\infty}\int_{B_{\lambda}(z_0)} |u_m-u_0|^2\mathrm{d}\gamma
\geq
   \varlimsup\limits_{m\rightarrow\infty}\int_{\left\{x\in B_{\lambda}(z_0):u_m(x)=0\right\}} |u_m-u_0|^2\mathrm{d}\gamma
\\&
\geq C^2\varlimsup\limits_{m\rightarrow\infty}|\left\{x\in B_{\lambda}(z_0):u_m(x)=0\right\}|_{\mathrm{d}\gamma}
\geq \frac{C^2\varepsilon}{2}\lim\limits_{m\rightarrow\infty}|B_{1}(z_m)|_{\mathrm{d}\gamma}
\\&=\frac{C^2\varepsilon}{2}|B_{1}(z_0)|_{\mathrm{d}\gamma}>0,
\end{align*}
which is a contradiction. This completes the proof.
\end{proof}
Using the weighted Sobolev inequality (Lemma \ref{SOBOINETG}) and
the classic Moser's iteration, we can prove the following boundedness result by standard arguments (see \cite[Theorem 4.1]{HANLIN}).
\begin{lemma}\label{sjgjloc} \ Let $0<\beta<\frac{N-2}{2}$. Suppose that
$0\leq u\in W_{\beta}^{1,2}(B_{1}(z))\cap L^{\infty}(B_{1}(z))$ satisfies
$$
-\Delta_{\beta} u
\leq 0
~~~{\rm in}~B_{1}(z),
$$
in the sense that
$$
\int_{B_{1}(z)}\frac{\nabla u \cdot \nabla \varphi }{|x|^{2\beta}}\mathrm{d}x
\leq 0~~~{\rm for~ all}~0\leq \varphi\in C_{c}^{\infty}(B_{1}(z)).
$$
Then there holds
$$
|u|_{\infty,B_{1/2}(z)}\leq \hat{C}(N,\beta)\left(\frac{1}{|B_{1}(z)|_{\mathrm{d}\gamma}}
\int_{B_{1}(z)} u^2\mathrm{d}\gamma\right)^{1/2}.
$$
\end{lemma}
Inspired by the ideas in \cite[Theorem 4.9]{HANLIN}, we can establish the following local estimate for supsolutions.

\begin{lemma}\label{lesgyjqfc} \ Let $0<\varepsilon\leq 1$, $0<\beta<\frac{N-2}{2}$ and $k>0$. Suppose that
$0\leq u\in W_{\beta}^{1,2}(B_{2\rho}(z))\cap L^{\infty}(B_{2\rho}(z))$ satisfies
$$
-\Delta_{\beta} u
\geq 0
~~~{\rm in}~B_{2\rho}(z),
$$
in the sense that
$$
\int_{B_{2\rho}(z)}\frac{\nabla u \cdot \nabla \varphi }{|x|^{2\beta}}\mathrm{d}x
\geq 0~~~{\rm for~ all}~0\leq \varphi\in C_{c}^{\infty}(B_{2\rho}(z)).
$$
If
$$
|\left\{x\in B_{\rho}(z):u(x)\geq k\right\}|_{\mathrm{d}\gamma}
\geq
\varepsilon
| B_{\rho}(z)|_{\mathrm{d}\gamma},
$$
then there holds
$$
\inf_{B_{\rho/2}(z)} u \geq C(N, \beta, \varepsilon)k.
$$
\end{lemma}

\begin{proof}[\bf Proof.]
By the scaling transformation $v(x)=\frac{1}{k}u(\rho x)$, $x\in B_2(\frac{z}{\rho})$,
we may assume that $\rho=k=1$.
Let $0<\delta<1$, $w=\left(\ln (u+\delta)\right)^-$. Then $w\in W_{\beta}^{1,2}(B_{1}(z))$,
$0\leq w\leq -\ln \delta$
and $-\Delta_{\beta} w \leq 0$ in $B_{1}(z)$
(see e.g., \cite[Lemma 4.6]{HANLIN}). Therefore, Proposition \ref{sjgjloc} yields
\begin{equation} \label{pyzydlocgj}
|w|_{\infty,B_{1/2}(z)}\leq \hat{C}(N,\beta)\left(\frac{1}{|B_{1}(z)|_{\mathrm{d}\gamma}}
\int_{B_{1}(z)} w^2\mathrm{d}\gamma\right)^{1/2}.
\end{equation}
Let $\eta \in C_{c}^{\infty}(B_{2}(z))$ be a cut-off function satisfying $\eta=1$ for
$x\in B_{1}(z)$ and $|\nabla \eta| \leq 2$. Set test function as
$\varphi=\frac{\eta^2}{u+\delta}$. Then we have
\begin{equation} \label{tfcautid}
\aligned
\int_{B_{2}(z)} \frac{\eta^2|\nabla u|^2}{(u+\delta)^2}\mathrm{d}\gamma
&\leq
2\int_{B_{2}(z)} \frac{\eta\nabla u\cdot\nabla \eta}{u+\delta}\mathrm{d}\gamma
\\&\leq
\frac{1}{2}\int_{B_{2}(z)} \frac{\eta^2|\nabla u|^2}{(u+\delta)^2}\mathrm{d}\gamma
+2\int_{B_{2}(z)} |\nabla \eta|^2\mathrm{d}\gamma.
\endaligned
\end{equation}
Note that $\nabla \ln (u+\delta)=\frac{\nabla u}{u+\delta}$. It follows
from (\ref{tfcautid}) that
\begin{equation} \label{eindsdwe}
\aligned
\int_{B_{1}(z)}|\nabla w|^2 \mathrm{d}\gamma
&\leq
\int_{B_{1}(z)}|\nabla \ln (u+\delta)|^2 \mathrm{d}\gamma
\leq
\int_{B_{2}(z)} \frac{\eta^2|\nabla u|^2}{(u+\delta)^2}\mathrm{d}\gamma
\\&\leq
4\int_{B_{2}(z)} |\nabla \eta|^2\mathrm{d}\gamma
\leq
16|B_{2}(z)|_{\mathrm{d}\gamma}.
\endaligned
\end{equation}
Since
$$|\left\{x\in B_{1}(z):w(x)=0\right\}|_{\mathrm{d}\gamma}
\geq
|\left\{x\in B_{1}(z):u(x)\geq 1\right\}|_{\mathrm{d}\gamma}
\geq
\varepsilon
| B_{1}(z)|_{\mathrm{d}\gamma},$$ we have from
(\ref{pyzydlocgj}), (\ref{eindsdwe}), Proposition \ref{posoin} and
Remark \ref{MUDOUTJ} that,
\begin{align*}
|w|_{\infty,B_{1/2}(z)}
&\leq
\frac{\hat{C}(N,\beta)}{|B_{1}(z)|_{\mathrm{d}\gamma}^{1/2}}
   \left(\int_{B_{1}(z)} w^2\mathrm{d}\gamma\right)^{1/2}
\leq
  \frac{C_1(N,\beta,\varepsilon)}{|B_{1}(z)|_{\mathrm{d}\gamma}^{1/2}}
   \left(\int_{B_{1}(z)} |\nabla w|^2\mathrm{d}\gamma\right)^{1/2}
\\&
  \leq
  C_2(N,\beta,\varepsilon)\frac{|B_{2}(z)|_{\mathrm{d}\gamma}^{1/2}}{|B_{1}(z)|_{\mathrm{d}\gamma}^{1/2}}
  \leq
  C_3(N,\beta,\varepsilon).
\end{align*}
Let $\delta \rightarrow 0$. We can get
$$
\inf_{B_{1/2}(z)} u \geq e^{-C_3(N, \beta, \varepsilon)}.
$$
\end{proof}

\begin{corollary}\label{tuilunse} \ Let $0<\varepsilon\leq 1$, $0<\beta<\frac{N-2}{2}$ and $k>0$. Suppose that
$0\leq u\in D_{\beta}^{1,2}(\mathbb{R}^{N})\cap L^{\infty}(\mathbb{R}^{N})$ satisfies
$$
-\Delta_{\beta} u
\geq 0
~~~{\rm in}~\mathbb{R}^{N},
$$
which means
$$
\int_{\mathbb{R}^{N}}\frac{\nabla u \cdot \nabla \varphi }{|x|^{2\beta}}\mathrm{d}x
\geq 0~~~{\rm for~ all}~0\leq \varphi\in C_{c}^{\infty}(\mathbb{R}^{N}).
$$
If
$$
|\left\{x\in B_{3\rho}(z):u(x)\geq k\right\}|_{\mathrm{d}\gamma}
\geq
\varepsilon
| B_{\rho}(z)|_{\mathrm{d}\gamma},
$$
then there holds
$$
\inf_{B_{3\rho}(z)} u \geq C(N, \beta, \varepsilon)k.
$$
\end{corollary}
\begin{proof}[\bf Proof.]
By Remark \ref{MUDOUTJ}, we see
\begin{align*}
|\left\{x\in B_{6\rho}(z):u(x)\geq k\right\}|_{\mathrm{d}\gamma}
&\geq
|\left\{x\in B_{3\rho}(z):u(x)\geq k\right\}|_{\mathrm{d}\gamma}
\geq
\varepsilon
| B_{\rho}(z)|_{\mathrm{d}\gamma}
\\&\geq
\frac{\varepsilon}{\bar{C}(N,\beta)^3}| B_{6\rho}(z)|_{\mathrm{d}\gamma},
\end{align*}
where $\bar{C}(N,\beta)$ is the doubling constant of $\gamma$.
Then Lemma \ref{lesgyjqfc} implies the result.
\end{proof}

The following Krylov-Safonov covering theorem on a doubling measure space plays a crucial role
in proving the Harnack inequality.
\begin{lemma}\label{fugailyl} {\rm(\cite{KINSHAN})} \ Let $E\subset B_R(z)$ be
a measurable set and  $0<\delta<1$. Define
$$
[E]_\delta=\bigcup_{\rho>0}\left\{B_R(z)\cap B_{3\rho}(x):x\in B_R(z),
|E\cap B_{3\rho}(x)|_{\mathrm{d}\gamma}
>\delta |B_{\rho}(x)|_{\mathrm{d}\gamma}
\right\}
.$$
Then, either $[E]_\delta=B_R(z)$, or else $$|[E]_\delta|_{\mathrm{d}\gamma}
\geq\frac{|E|_{\mathrm{d}\gamma}}{\delta\bar{C}(N,\beta)}.$$
\end{lemma}

Using the same arguments as in the proof of \cite[Theorem 7.1]{KINSHAN}, as a
consequence of Remark \ref{MUDOUTJ}, Corollary \ref{tuilunse} and
Lemma \ref{fugailyl}, we have the following weak Harnack inequality.

\begin{lemma}\label{wehncbequ} \ Let $0<\beta<\frac{N-2}{2}$. Suppose that
$0\leq v\in D_{\beta}^{1,2}(\mathbb{R}^{N})\cap L^{\infty}(\mathbb{R}^{N})$ satisfies
$$
-\Delta_{\beta} v
\geq 0
~~~{\rm in}~\mathbb{R}^{N}.
$$
Then there are $\sigma(N,\beta)>0$ and $\tilde{C}(N,\beta)>0$ such that
$$
\inf_{B_{3R}(z)} v \geq \tilde{C}(N,\beta)\left(\frac{1}{|B_{R}(z)|_{\mathrm{d}\gamma}}
\int_{B_{R}(z)} v^{\sigma(N,\beta)}\mathrm{d}\gamma\right)^{1/\sigma(N,\beta)}.
$$
\end{lemma}

\begin{proof}[\bf Proof of Theorem \ref{AB}.] Set $v=|x|^\beta u$.
Then by Remark \ref{EQTRREMR}, we have that
$v \in D_{\beta}^{1,2}(\mathbb{R}^{N})$
weakly solves (\ref{chodjiadfjubeq}). And Lemma \ref{Linfesti} implies
$v\in L^{\infty}(\mathbb{R}^{N})$. Further, according to Lemma \ref{wehncbequ}, we see that for any $R>0$,
$$
\inf_{B_{R}(0)} v \geq \tilde{C}(N,\beta)\left(\frac{1}{|B_{R}(0)|_{\mathrm{d}\gamma}}
\int_{B_{R}(0)} v^{\sigma(N,\beta)}\mathrm{d}\gamma\right)^{1/\sigma(N,\beta)},
$$
which leads to $\inf_{B_{R}(0)} v>0$. Thus, there is a positive number $c_1$ such that
$u(x)\geq c_1|x|^{-\beta}$ in $B_{R}(0)$. Then by $v\in L^{\infty}(\mathbb{R}^{N})$, we obtain
\begin{equation} \label{einingujboths}
\frac{c_1}{|x|^{\beta}}\leq u \leq \frac{c_2}{|x|^{\beta}}~~~{\rm in}~B_{R}(0).
\end{equation}
Denote $K_u$ the Kelvin transform of $u$, that is,
$$K_u(x)=\frac{1}{|x|^{N-2}}u\left(\frac{x}{|x|^2}\right).$$
By \cite[Lemma 2.1]{GUOLUN}, we see that
$$\Delta K_u(x)=\frac{1}{|x|^{N+2}}\Delta u\left(\frac{x}{|x|^2}\right)$$
and
$$\left((I_{\alpha}\ast K_u^{\bar{p}})K_u^{\bar{p}-1}\right)(x)
=\frac{1}{|x|^{N+2}}\left((I_{\alpha}\ast u^{\bar{p}})u^{\bar{p}-1}\right)\left(\frac{x}{|x|^2}\right).
$$
Hence
$$-\Delta K_u-\frac{\vartheta}{|x|^2}K_u
=(I_{\alpha}\ast K_u^{\bar{p}})K_u^{\bar{p}-1}~~~{\rm in}~\mathbb{R}^{N}.
$$
From the previous discussion, for any $R'>0$, there exist two positive numbers $c_3$ and $c_4$
such that
$$\frac{c_3}{|x|^{\beta}}\leq K_u(x) \leq \frac{c_4}{|x|^{\beta}}~~~{\rm in}~B_{R'}(0).$$
Consequently,
$$\frac{c_3}{|x|^{N-2-\beta}}\leq u(x) \leq \frac{c_4}{|x|^{N-2-\beta}}~~~{\rm in}~\mathbb{R}^{N}\setminus B_{1/R'}(0),$$
which, together with (\ref{einingujboths}), yields
$$
\frac{c}{|x|^{\beta}\left(1+|x|^{2-\frac{4\beta}{N-2}}\right)^{\frac{N-2}{2}}}
\leq u(x) \leq \frac{C}{|x|^{\beta}\left(1+|x|^{2-\frac{4\beta}{N-2}}\right)^{\frac{N-2}{2}}}.
$$
\end{proof}

\section{Symmetry of solutions and proof of Theorem \ref{SYM}}
In this section, we shall use the moving plane method in integral form to prove symmetry
in each given direction. For this purpose, we first recall some notations.
For $\lambda\in \mathbb{R}$ and $x=(x_1,x_2,\cdots,x_N)\in \mathbb{R}^{N}$, we define
$$\Sigma_\lambda=\{x\in \mathbb{R}^{N}:x_1<\lambda\},$$
and let $x_\lambda=(2\lambda-x_1,x_2,\cdots,x_N)$ and $u_\lambda(x)=u(x_\lambda)$.
We set
$$\Sigma_\lambda^u=\{x\in \Sigma_\lambda:u(x)>u_\lambda(x)\},
~~~\Sigma_\lambda^v=\{x\in \Sigma_\lambda:v(x)>v_\lambda(x)\}.$$

\begin{proof}[\bf Proof of Theorem \ref{SYM}.] Let $\lambda<0$,
$w_\lambda=(u_\lambda-u)^{-}\chi_{\Sigma_\lambda}$. Note that
$w_\lambda=0$ on $\partial \Sigma_\lambda$. Then $w_\lambda \in D^{1,2}(\mathbb{R}^{N})$.
Set $v=I_{\alpha}\ast u^{\bar{p}}$. The Hardy-Littlewood-Sobolev inequality gives $v\in L^{\frac{2N}{N-\alpha}}(\mathbb{R}^{N})$. Since
$$
-\Delta u_\lambda-\frac{\vartheta}{|x_\lambda|^2}u_\lambda
=(I_{\alpha}\ast u^{\bar{p}})_\lambda u^{\bar{p}-1}_\lambda
~~~{\rm in}~\mathbb{R}^{N},
$$
we have that
\begin{equation} \label{eguotest}
\aligned
&-\Delta (u_\lambda-u)-\frac{\vartheta}{|x|^2}(u_\lambda-u)
+\vartheta u_\lambda\left(\frac{1}{|x|^2}-\frac{1}{|x_\lambda|^2}\right)
\\&~~~~~~~~=(v_\lambda-v)u^{\bar{p}-1}+v_\lambda\left(u^{\bar{p}-1}_\lambda-u^{\bar{p}-1}\right)
~~~{\rm in}~\mathbb{R}^{N}.
\endaligned
\end{equation}
Testing (\ref{eguotest}) with $w_\lambda$, we obtain
\begin{align*}
&\int_{\Sigma_\lambda^u}\nabla (u_\lambda-u) \nabla w_\lambda\mathrm{d}x
   -\vartheta\int_{\Sigma_\lambda^u}\frac{w_\lambda(u_\lambda-u)}{|x|^2}\mathrm{d}x
    +\vartheta\int_{\Sigma_\lambda^u} u_\lambda w_\lambda
      \left(\frac{1}{|x|^2}-\frac{1}{|x_\lambda|^2}\right)\mathrm{d}x
\\&~~~~~~~~
=\int_{\Sigma_\lambda^u}(v_\lambda-v)u^{\bar{p}-1}w_\lambda\mathrm{d}x
+\int_{\Sigma_\lambda^u}v_\lambda w_\lambda\left(u^{\bar{p}-1}_\lambda-u^{\bar{p}-1}\right)\mathrm{d}x.
\end{align*}
Note that $\bar{p}>2$ and $|x|>|x_\lambda|$ on $\Sigma_\lambda^u$. It follows from the Hardy inequality and
the Sobolev inequality that
\begin{equation} \label{varfainetest}
\aligned
&\frac{1}{(S(N))^2}\left(1-\frac{4\vartheta}{(N-2)^2}\right)|w_\lambda|_{2^*,\Sigma_\lambda^u}^2
\\&~~~
\leq
\int_{\Sigma_\lambda^u} |\nabla w_\lambda|^2 \mathrm{d}x
-\vartheta\int_{\Sigma_\lambda^u} \frac{w_\lambda^2}{|x|^2}\mathrm{d}x
\\&~~~
\leq
\int_{\Sigma_\lambda^u}(v-v_\lambda)u^{\bar{p}-1}w_\lambda\mathrm{d}x
+\int_{\Sigma_\lambda^u}v_\lambda w_\lambda\left(u^{\bar{p}-1}-u^{\bar{p}-1}_\lambda\right)\mathrm{d}x
\\&~~~
\leq
\int_{\Sigma_\lambda^u\cap \Sigma_\lambda^v}(v-v_\lambda)u^{\bar{p}-1}w_\lambda\mathrm{d}x
+(\bar{p}-1)\int_{\Sigma_\lambda^u}v_\lambda u^{\bar{p}-2} w_\lambda^2\mathrm{d}x
\\&~~~
\leq
|v-v_\lambda|_{\frac{2N}{N-\alpha},\Sigma_\lambda^v}
|u|_{2^*,\Sigma_\lambda^u}^{\bar{p}-1}
|w_\lambda|_{2^*,\Sigma_\lambda^u}
+(\bar{p}-1)|v_\lambda|_{\frac{2N}{N-\alpha},\Sigma_\lambda^u}
|u|_{2^*,\Sigma_\lambda^u}^{\bar{p}-2}
|w_\lambda|_{2^*,\Sigma_\lambda^u}^2,
\endaligned
\end{equation}
where $S(N)$ is the best Sobolev constant for the embedding
$D^{1,2}(\mathbb{R}^{N})\hookrightarrow L^{2^*}(\mathbb{R}^{N})$.
Since $|x-y|<|x_\lambda-y|$ for $x\in \Sigma_\lambda^v$
and $y\in \Sigma_\lambda^u$, we have that for $x\in \Sigma_\lambda^v$,
\begin{equation} \label{varfaineseta}
\aligned
0&<v(x)-v_\lambda(x)
\\&=\int_{\Sigma_\lambda}\frac{u^{\bar{p}}(y)}{|x-y|^{N-\alpha}}\mathrm{d}y+
   \int_{\Sigma_\lambda}\frac{u_\lambda^{\bar{p}}(y)}{|x_\lambda-y|^{N-\alpha}}\mathrm{d}y
\\&~~~~~~~~-\int_{\Sigma_\lambda}\frac{u^{\bar{p}}(y)}{|x_\lambda-y|^{N-\alpha}}\mathrm{d}y
   -\int_{\Sigma_\lambda}\frac{u_\lambda^{\bar{p}}(y)}{|x-y|^{N-\alpha}}\mathrm{d}y
\\&
=\int_{\Sigma_\lambda}\left(\frac{1}{|x-y|^{N-\alpha}}-\frac{1}{|x_\lambda-y|^{N-\alpha}}\right)
  \left(u^{\bar{p}}(y)-u^{\bar{p}}_\lambda(y)\right)\mathrm{d}y
\\&
\leq
\int_{\Sigma_\lambda^u}\left(\frac{1}{|x-y|^{N-\alpha}}-\frac{1}{|x_\lambda-y|^{N-\alpha}}\right)
  \left(u^{\bar{p}}(y)-u^{\bar{p}}_\lambda(y)\right)\mathrm{d}y
\\&
\leq
\int_{\Sigma_\lambda^u}\frac{u^{\bar{p}}(y)-u^{\bar{p}}_\lambda(y)}{|x-y|^{N-\alpha}}\mathrm{d}y.
\endaligned
\end{equation}
Therefore, the Hardy-Littlewood-Sobolev inequality implies that
\begin{equation} \label{varfaisyag}
\aligned
|v-v_\lambda|_{\frac{2N}{N-\alpha},\Sigma_\lambda^v}
&\leq
     \left|\int_{\Sigma_\lambda^u}\frac{u^{\bar{p}}(y)-u^{\bar{p}}_\lambda(y)}
        {|x-y|^{N-\alpha}}\mathrm{d}y\right|_{\frac{2N}{N-\alpha},\Sigma_\lambda^v}
\leq
  S(N,\alpha)|u^{\bar{p}}-u^{\bar{p}}_\lambda|_{\frac{2N}{N+\alpha},\Sigma_\lambda^u}
\\&
\leq
  S(N,\alpha)\bar{p}|u|_{2^*,\Sigma_\lambda^u}^{\frac{\alpha+2}{N-2}}|w_\lambda|_{2^*,\Sigma_\lambda^u}.
\endaligned
\end{equation}
It follows from (\ref{varfainetest}) and (\ref{varfaisyag}) that
\begin{equation} \label{fsgjeinfainetest}
\aligned
&\frac{1}{(S(N))^2}\left(1-\frac{4\vartheta}{(N-2)^2}\right)|w_\lambda|_{2^*,\Sigma_\lambda^u}^2
\\&~~~
\leq
S(N,\alpha)\left(\bar{p}|u|_{2^*,\Sigma_\lambda^u}^{\frac{2\alpha+4}{N-2}}
   +(\bar{p}-1)|u|_{2^*}^{\bar{p}}|u|_{2^*,\Sigma_\lambda^u}^{\bar{p}-2}\right)
   |w_\lambda|_{2^*,\Sigma_\lambda^u}^2.
\endaligned
\end{equation}
Since $|u|_{2^*,\Sigma_\lambda^u}\rightarrow 0$ as $\lambda \rightarrow -\infty$,
there is $\bar{\lambda}<0$ such that for $\lambda<\bar{\lambda}$,
$$|w_\lambda|_{2^*,\Sigma_\lambda^u}=0.$$
Therefore, $|\Sigma_\lambda^u|=|\Sigma_\lambda^v|=0$. Since $u \in C(\mathbb{R}^{N}\setminus\{0\})$.
We deduce that for $\lambda<\bar{\lambda}$,
$$u_\lambda\geq u~~~\mbox{and}~~~v_\lambda\geq v ~~~\mbox{on}~\Sigma_\lambda\setminus\{0_\lambda\}.$$
Thus, (\ref{eguotest}) implies that
$$
-\Delta (u_\lambda-u)\geq 0~~~\mbox{in}~\Sigma_\lambda\setminus\{0_\lambda\}.
$$
We claim that $u_\lambda > u$ on $\Sigma_\lambda\setminus\{0_\lambda\}$.
If not, $u_\lambda(x_0)=u(x_0)$ for some $x_0\in \Sigma_\lambda\setminus\{0_\lambda\}$, then by
the strong maximum principle, we see $u_\lambda=u$ on $\Sigma_\lambda\setminus\{0_\lambda\}$. From
Theorem \ref{AB}, we have
$$
\lim\limits_{x\rightarrow 0_\lambda} (u_\lambda(x)-u(x))=\infty,
$$
which is a contradiction.

Let $$\lambda^*=\sup\{\bar{\lambda}<0:u_\lambda >u~\mbox{on}~
 \Sigma_\lambda\setminus\{0_\lambda\},~\forall \lambda<\bar{\lambda} \}.$$
Next we show that $\lambda^*=0$. Arguing indirectly, suppose that $\lambda^*<0$.
Since $u \in C(\mathbb{R}^{N}\setminus\{0\})$, we know that $u_{\lambda^*} \geq u$
on $\Sigma_{\lambda^*}\setminus\{0_{\lambda^*}\}$. Therefore, similar to (\ref{varfaineseta}),
for $x\in \Sigma_{\lambda^*}$,
$$
v(x)-v_{\lambda^*}(x)=\int_{\Sigma_{\lambda^*}}\left(\frac{1}{|x-y|^{N-\alpha}}-\frac{1}{|x_{\lambda^*}-y|^{N-\alpha}}\right)
  \left(u^{\bar{p}}(y)-u^{\bar{p}}_{\lambda^*}(y)\right)\mathrm{d}y\leq 0.
$$
Thus, (\ref{eguotest}) leads to
$$-\Delta (u_{\lambda^*}-u)\geq 0~~~\mbox{in}~\Sigma_{\lambda^*}\setminus\{0_{\lambda^*}\}.$$
Hence, the strong maximum principle and Theorem \ref{AB} give
\begin{equation} \label{flamiceshiein}
u_{\lambda^*} > u~~~\mbox{on}~\Sigma_{\lambda^*}\setminus\{0_{\lambda^*}\},
\end{equation}
So, we have that as $\lambda \rightarrow \lambda^*+0$,
$|u|^{2^*}\chi_{\Sigma_\lambda^u}\rightarrow 0$ a.e. on $\mathbb{R}^{N}$.
It follows from Lebesgue dominated theorem that
$$\lim\limits_{\lambda \rightarrow \lambda^*+0}\int_{\Sigma_\lambda^u}|u|^{2^*} \mathrm{d}x=0,
 $$
From the above and (\ref{fsgjeinfainetest}), there is $\varepsilon_0>0$
such that for $\lambda^*<\lambda<\lambda^*+\varepsilon_0$,
$$|w_\lambda|_{2^*,\Sigma_\lambda^u}=0.$$
Thus, $u_\lambda \geq u$ on $\Sigma_\lambda\setminus\{0_\lambda\}$.
Similar to (\ref{flamiceshiein}), for $\lambda^*<\lambda<\lambda^*+\varepsilon_0$, we have
$$
u_\lambda > u~~~\mbox{on}~\Sigma_\lambda\setminus\{0_\lambda\}.
$$
This contradicts the definition of $\lambda^*$. Therefore, $\lambda^*=0$.

Finally, since the direction of $x_1$ can be chosen arbitrarily, we obtain that $u(x)$
must radially symmetric and decreasing about the origin.
\end{proof}

\section*{Acknowledgments}  This work was supported by the National Natural Science Foundation of China (Grant Nos. 12101273, 11971485),
 and the Foundation of Education Department of Jiangxi Province (No. GJJ200861).

\section*{Data Availability Statement}  The data that supports the findings of this study are available within the article [and its supplementary material].

\end{document}